\documentclass[a4paper]{amsart}
\usepackage{tikz}
\usetikzlibrary{matrix}
\usetikzlibrary{decorations.pathreplacing}
\usetikzlibrary{decorations.markings}
\usetikzlibrary{arrows}
\usetikzlibrary{calc}
\usetikzlibrary{shapes.misc}
\usetikzlibrary{fit}
\usepgflibrary{decorations.pathmorphing}
\usepgflibrary{shapes.geometric}
\usepackage[colorlinks, linkcolor=black, citecolor=blue]{hyperref}

\newcommand*\arrowoffset{0.4pt}
\makeatletter
\pgfarrowsdeclare{new double arrowhead}{new double arrowhead}
{
  \pgfutil@tempdima=-0.84pt%
  \advance\pgfutil@tempdima by-1.3\pgflinewidth%
  \pgfutil@tempdimb=0.21pt%
  \advance\pgfutil@tempdimb by.625\pgflinewidth%
  \pgfarrowsleftextend{+\pgfutil@tempdima}
  \advance\pgfutil@tempdimb by \arrowoffset 
  \pgfarrowsrightextend{+\pgfutil@tempdimb}
}
{
  \pgfsetstrokecolor{black}
  \pgfsetlinewidth{0.6pt}
  \pgfutil@tempdima=0.28pt%
  \advance\pgfutil@tempdima by.3\pgflinewidth%
  \pgfsetlinewidth{0.8\pgflinewidth}
  \pgfsetdash{}{+0pt}
  \pgfsetroundcap
  \pgfsetroundjoin
  \pgfpathmoveto{\pgfpoint{-3\pgfutil@tempdima+\arrowoffset}{4\pgfutil@tempdima}}
  \pgfpathcurveto
                {\pgfpoint{-2.75\pgfutil@tempdima+\arrowoffset}{2.5\pgfutil@tempdima}}
                {\pgfpoint{\arrowoffset}{0.25\pgfutil@tempdima}}
                {\pgfpoint{0.75\pgfutil@tempdima+\arrowoffset}{0pt}}
  \pgfpathcurveto
                {\pgfpoint{\arrowoffset}{-0.25\pgfutil@tempdima}}
                {\pgfpoint{-2.75\pgfutil@tempdima+\arrowoffset}{-2.5\pgfutil@tempdima}}
                {\pgfpoint{-3\pgfutil@tempdima+\arrowoffset}{-4\pgfutil@tempdima}}
  \pgfusepathqstroke
}
\pgfarrowsdeclare{new single arrowhead}{new single arrowhead}
{
  \pgfutil@tempdima=-0.84pt%
  \advance\pgfutil@tempdima by-1.3\pgflinewidth%
  \pgfutil@tempdimb=0.21pt%
  \advance\pgfutil@tempdimb by.625\pgflinewidth%
  \pgfarrowsleftextend{+\pgfutil@tempdima}
  \advance\pgfutil@tempdimb by \arrowoffset 
  \pgfarrowsrightextend{+\pgfutil@tempdimb}
}
{
  \pgfsetstrokecolor{black}
  \pgfsetlinewidth{0.6pt}
  \pgfutil@tempdima=0.28pt%
  \advance\pgfutil@tempdima by.3\pgflinewidth%
  \pgfsetlinewidth{0.8\pgflinewidth}
  \pgfsetdash{}{+0pt}
  \pgfsetroundcap
  \pgfsetroundjoin
  \pgfpathmoveto{\pgfpoint{-3\pgfutil@tempdima}{4\pgfutil@tempdima}}
  \pgfpathcurveto
                {\pgfpoint{-2.75\pgfutil@tempdima}{2.5\pgfutil@tempdima}}
                {\pgfpoint{0pt}{0.25\pgfutil@tempdima}}
                {\pgfpoint{0.75\pgfutil@tempdima}{0pt}}
  \pgfpathcurveto
                {\pgfpoint{0pt}{-0.25\pgfutil@tempdima}}
                {\pgfpoint{-2.75\pgfutil@tempdima}{-2.5\pgfutil@tempdima}}
                {\pgfpoint{-3\pgfutil@tempdima}{-4\pgfutil@tempdima}}
  \pgfusepathqstroke
}
\makeatother

\tikzset{morphlabel/.style={draw=black, thin, rectangle, minimum width=7pt, fill=white, font=\scriptsize}}

\tikzset{double arrow scope/.style={every path/.style={double, -new double arrowhead}}}

\newcommand\selectpart[2][\selectcolour]{\node [draw, fit=#2, inner sep=0.8*\cobordismlinewidth, #1, line width=\cobordismlinewidth] {};}

\newlength{\cobordismlinewidth}	

\newcommand\xdoubleto[1]{\mathbin{\begin{tikzpicture}[baseline={([yshift=-3pt]
current bounding box.south)}]
    \node (A) at (0,0) [inner xsep=0pt, inner ysep=1pt, minimum width=0.2cm] {\ensuremath{ #1 \strut}};
    \draw [double,-new double arrowhead]
        ([xshift=-2.5pt] A.south west)
        to ([xshift=3pt] A.south east);
\end{tikzpicture}}}

\theoremstyle{plain}
\newtheorem{theorem}{Theorem}
\newtheorem{lemma}[theorem]{Lemma}
\newtheorem{proposition}[theorem]{Proposition}

\theoremstyle{definition}
\newtheorem{defn}[theorem]{Definition}

\newenvironment{tz}[1][]{\begin{aligned}\begin{tikzpicture}[#1]}{\end{tikzpicture}\end{aligned}}
\newcommand{\id}{\ensuremath{\mathrm{id}}}
\newcommand{\bicat}[1]{\ensuremath{\mathbf{#1}}}
\newcommand\Hom{\ensuremath{\mathrm{Hom}}}

\begin{document}

\title{Quasistrict symmetric monoidal 2-categories via wire diagrams}

\author[B. Bartlett]{Bruce Bartlett}
\address{Department of Mathematics, University of Stellenbosch}
\curraddr{Mathematical Institute, University of Oxford}
\email{bartlett@maths.ox.ac.uk}

\tikzset{every picture/.style={yscale=0.7}}

\begin{abstract} In this paper we give an expository account of quasistrict symmetric monoidal 2-categories as introduced by Schommer-Pries. We reformulate the definition using a graphical calculus called wire diagrams, which facilitates computations and emphasizes the central role played by the interchangor coherence isomorphisms.

\end{abstract}
\maketitle

\section{Introduction}

Establishing the definition of a symmetric monoidal bicategory, and proving associated coherence and strictification results, has been a considerable effort by a number of authors~\cite{kv94-bm2, kv94-2categories, BN96, ds97-mbh, cr98-gcb, bl98-2t, GPS95, gurskithesis, gur11-lsc, go13-ils, g13-ctd, m00-bc, st13-ccb}; see also the references in~\cite{CSPthesisLatest}. Recently, Schommer-Pries has defined a stricter version of a symmetric monoidal bicategory, called a {\em quasistrict symmetric monoidal 2~-category}, and proved the following strictification result:

\begin{theorem}[\cite{CSPthesisLatest}] Every symmetric monoidal bicategory is equivalent to a quasistrict symmetric monoidal 2~-category. \label{csp_qs_theorem}
\end{theorem}
\noindent In this paper we give an expository account of this result, by introducing  a graphical notation which we call {\em wire diagrams}. The utility of this notation is twofold. Firstly, wire diagrams offer a simple visual explanation for what is going on. Secondly, wire diagrams facilitate {\em working} with these structures and making actual {\em computations}. In fact, this was the motivation for the coherence result above. As part of a project related to three~-dimensional topological quantum field theory, we found ourselves working in a symmetric monoidal bicategory presented by generators and relations~\cite{PaperI, PaperII, PaperIII}. The calculations involved were all expressed in this graphical calculus, and it would have been intractable to perform them without it. 

Ordinary algebra is about manipulating a string of symbols on a line. One can think of algebraic manipulations in a symmetric monoidal bicategory as being a form of {\em stable 3~-dimensional algebra}. Wire diagrams are one possible notation for this. The basic idea is that the tensor product direction runs out of the page, composition of 1\-morphisms and (horizontal) composition of 2\-morphisms runs up the page, and (vertical) composition of 2\-morphisms runs from left-to-right\footnote{Unfortunately, what is usually called {\em vertical} composition $\circ$ of 2\-morphisms runs horizontally in wire diagrams, and what is usually called {\em horizontal} composition $\ast$ runs vertically! }:

\[
\begin{tz}[scale=1.3]
	\draw[->] (0,0) to (0,1)  node[above] {\footnotesize 1\-morphisms} ;
	\draw[->] (0,0) to  (-150:1) node[below] {\footnotesize tensor product};
	\draw[->] (0,0) to node[xshift=-0.3cm, below right] {\footnotesize 2\-morphisms} (1,0);
\end{tz}
\qquad \begin{aligned}
\includegraphics{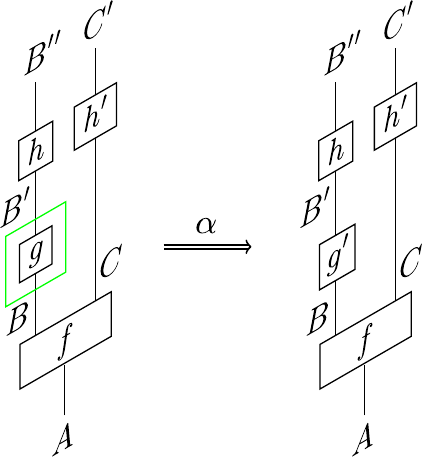} \end{aligned}
\]
To make such a diagram clearer, it will usually just be drawn flat in the page (but the three-dimensional picture should be kept in mind), like this:
\[
\begin{tz}[xscale=3, yscale=5.5]
\node (1) at (0,0)
{
$\begin{tz}[xscale=0.7, yscale=0.6]
  \draw(0,0) to (0,2.5) node[morphlabel] (g) {$g$} to (0,4) node[morphlabel] (h) {$h$} to (0,5);
  \draw(1,0) to (1,4) node[morphlabel] {$h'$} to (1,5);
\node[morphlabel] at (0.5, 1) {$\,\,\,\, \,\,\, f \,\,\,\, \,\,\,$};
	\draw[green] ([xshift=-4pt, yshift=4pt] g.north west) rectangle ([xshift=4pt, yshift=-4pt] g.south east);
\end{tz}$
};

\node (2) at (1,0)
{
$\begin{tz}[xscale=0.7, yscale=0.6]
  \draw(0,0) to (0,2.5) node[morphlabel] (g') {$g'$} to (0,4) node[morphlabel] (h) {$h$} to (0,5);
  \draw(1,0) to (1,4) node[morphlabel] {$h'$} to (1,5);
\node[morphlabel]  at (0.5, 1) {$\,\,\,\, \,\,\, f \,\,\,\, \,\,\,$};
\end{tz}$
};
 
\begin{scope}[double arrow scope]
	\draw ([xshift=0.1em] 1.east) -- node[above] {$\alpha$} ([xshift=-0.1em] 2.west);
\end{scope}
\end{tz}
\]
The coloured box above indicates where the 2-morphism $\alpha$ is acting. We will explain this notation as we go along. 

Before the result of Schommer-Pries (Theorem \ref{csp_qs_theorem}), the most powerful strictification result for symmetric monoidal bicategories was the result of Gurski and Osorno~\cite{go13-ils}. They proved that every symmetric monoidal bicategory is equivalent to a {\em semistrict symmetric monoidal 2-category}. Besides the tensorator 2-isomorphisms 
\begin{equation} \label{t11}
  \Phi_{(f',g'),(f,g)} \colon (f'\otimes g') \circ (f \otimes g) \Rightarrow f'f \otimes g'g
\end{equation}
coming from the underlying semistrict monoidal 2-category, a semistrict symmetric monoidal 2-category has a host of additional coherence data: the braiding `naturators' $\beta_{f,g}$, and the braiding `bilinearators' $R_{A,B|C}$ and $S_{A | B,C}$ (see~\cite{go13-ils}). Theorem \ref{csp_qs_theorem} says that these latter coherence isomorphisms can be made into identities, at the cost of passing to an equivalent symmetric monoidal bicategory. 

To underscore this point of view, we will introduce a slimmed-down variant of the definition of a quasistrict symmetric monoidal 2-category, which we call a {\em stringent symmetric monoidal 2-category}. The `stringent' definition is equivalent to the `quasistrict' one, but does not explicitly contain redundant data, such as the braiding naturators $\beta_{f,g}$. Moreover it does not refer to the full 4-variable tensorator \eqref{t11}, but only to the underlying {\em interchangor} of $\Phi$, 
\[
\phi_{f,g} \colon (f \otimes \id_{B'}) \circ (\id_{A} \otimes g) \Rightarrow (\id_{A'} \otimes g) \circ (f \otimes \id_{B})
\]
drawn in wire diagrams as follows:
\[
\begin{tz}[xscale=1.3] 
\node (1) at (0,0)
{
$\begin{tz}
  \draw(0,0) node[below] {$A$} to (0,2) node[morphlabel] (f) {$f$} to (0,3) node[above] {$A'$};
  \draw(1,0) node[below] {$B$} to (1,1) node[morphlabel] (g) {$g$} to (1,3) node[above] {$B'$};
\end{tz}$
};  

\node (2) at (2,0)
{
$\begin{tz}
  \draw(0,0) node[below] {$A$} to (0,1) node[morphlabel] (f) {$f$} to (0,3) node[above] {$A'$};
  \draw(1,0) node[below] {$B$} to (1,2) node[morphlabel] (g) {$g$} to (1,3) node[above] {$B'$};
\end{tz}$
};  

\begin{scope}[double arrow scope]
	\draw (1) -- node[above] {$\phi_{f,g}$} (2);
\end{scope}

\end{tz}
\]
Slimming down the definition of quasistrict symmetric monoidal 2-category in this way makes it more suitable for a diagrammatic calculus, as well as, we hope, psychologically more pleasant. However, we view this distinction between the `stringent' and `quasistrict' forms of the definition as only a technical one, which is explicitly made in this paper for the purpose of precision; other authors may choose not to make this distinction, leaving it implicitly understood.

This paper is structured as follows. In Section \ref{wire_diag} we introduce wire diagrams in the familiar setting of 2-categories. In Section \ref{semistrict_sec} we review semistrict monoidal 2-categories. In Section \ref{stringent_sec} we introduce stringent monoidal 2-categories, extend the wire diagram notation to this setting, and prove that a stringent monoidal 2-category is the same thing as a semistrict monoidal 2-category. In Section \ref{stringent_sym_sec} we introduce stringent symmetric monoidal 2-categories, extend the wire diagram notation to this setting, and prove that a stringent symmetric monoidal 2-category is the same thing as a quasistrict symmetric monoidal 2-category. 

\subsection*{Notation} We will use the convention that `2-category' refers to a {\em strict} bicategory. Bicategories and 2-categories $\bicat{M}$ will be written in bold font and categories $C$ in plain font. 

\subsection*{Remark} The wire diagram notation can be extended in a straightforward way to give a natural graphical calculus for semistrict braided monoidal 2-categories (in the sense of~\cite{kv94-bm2, cr98-gcb, BN96, gur11-lsc} too, though we do not do this here.

\section{Wire diagrams for 2-categories} \label{wire_diag}
In this section we introduce wire diagrams in the setting of 2-categories. 

Let $\bicat{M}$ be a 2-category. The objects $A, B, \ldots$ of $\bicat{M}$ are drawn as:
 \[
  \begin{tz} \draw(0,0) node[below] {$A$} to (0,1) node[above] {$A$}; \end{tz} 
 \]
A 1-morphism $f \colon A \rightarrow B$ is drawn as:
 \[
  \begin{tz} \draw(0,0) node[below] {$A$} to  node[morphlabel, draw=black] {$f$} (0,2) node[above] {$B$}; \end{tz} 
 \] 
Note that composition of 1-morphisms runs from {\em bottom} to {\em top}! If $f,g\colon A \Rightarrow B$ are 1-morphisms, then a 2-morphism $\alpha  \colon f \Rightarrow g$ is drawn as:
\[
\begin{tz}[xscale=1.3] 
\node (1) at (0,0)
{
$\begin{tz}
  \draw(0,0) node[below] {$A$} to  node[morphlabel] {$f$} (0,2) node[above] {$B$}; 
\end{tz}$
};  

\node (2) at (1,0)
{
$\begin{tz}
  \draw(0,0) node[below] {$A$} to  node[morphlabel] {$g$} (0,2) node[above] {$B$}; 
\end{tz}$
};  

\begin{scope}[double arrow scope]
	\draw (1) -- node[above] {$\alpha$} (2);
\end{scope}
\end{tz}
 \] 
If $f \colon A \rightarrow B$ and $g \colon B \rightarrow C$ are 1-morphisms, then their composite $g \circ f : A \rightarrow C$ is drawn by stacking them on top of each other:
\[
\begin{tz}
  \draw(0,0) node[below] {$A$} to  node[morphlabel] {$g \circ f$} (0,3.5)  node[above] {$B$}; 
\end{tz}
\quad \equiv \quad
\begin{tz}
  \draw(0,0) node[below] {$A$} to (0,1) node[morphlabel] {$f$} to node[right] {$B$} (0,2.5) node[morphlabel] {$g$} to (0,3.5) node[above] {$C$};
\end{tz}
\]
Usually in a 2-category, we think of there being two composition laws for 2-morphisms: {\em horizontal} and {\em vertical} composition. In wire diagrams we will single out `vertical composition' as the primary operation (which we will just call {\em composition} of 2-morphisms for simplicity), and describe horizontal composition in terms of whiskering. So, if $\alpha \colon f \Rightarrow g$ and $\beta \colon g \Rightarrow h$ are 2-morphisms, then their composite $\beta \circ \alpha \colon f \Rightarrow h$ is drawn as:
\[
\begin{tz}[xscale=1.3] 
\node (1) at (0,0)
{
$\begin{tz}
  \draw(0,0) node[below] {$A$} to  node[morphlabel] {$f$} (0,2) node[above] {$B$}; 
\end{tz}$
};  

\node (2) at (1,0)
{
$\begin{tz}
  \draw(0,0) node[below] {$A$} to  node[morphlabel] {$g$} (0,2) node[above] {$B$}; 
\end{tz}$
};  

\node (3) at (2,0)
{
$\begin{tz}
  \draw(0,0) node[below] {$A$} to  node[morphlabel] {$h$} (0,2) node[above] {$B$}; 
\end{tz}$
};  

\begin{scope}[double arrow scope]
	\draw (1) -- node[above] {$\alpha$} (2);
	\draw (2) -- node[above] {$\beta$} (3);
\end{scope}
\end{tz}
\]
Suppose $A \stackrel{f_1}{\rightarrow} B \stackrel{f_2}{\rightarrow} C \stackrel{f_3} \rightarrow D$ are a composable triple of 1-morphisms, and that $\alpha \colon f_2 \Rightarrow g$ is a 2-morphism. Then as usual we can {\em whisker} $\alpha$ with the identity 2-morphisms on $f_1$ and $f_3$ respectively to obtain 
\[
\id_{f_3} \ast \alpha \ast \id_{f_1} : f_3 \circ f_2 \circ f_1 \Rightarrow f_3 \circ g \circ f_1.
\]
Whiskering is drawn by enclosing the source of the 2-morphism with a box. So, the diagram 
\[
\begin{tz}[xscale=1.5] \tikzset{every picture/.style={yscale=0.9}}
\node (1) at (0,0)
{
$\begin{tz}
  \draw(0,0) node[below] {$A$} to (0,1) node[morphlabel] {$f_1$} to node[right] {$B$} (0,2.5) node[morphlabel] (A) {$f_2$} to node[right] {$C$} (0,4) node[morphlabel] {$f_3$} to (0,5) node[above] {$D$};
  \selectpart[green, inner sep=2pt]{(A)};
\end{tz}$
};  

\node (2) at (1,0)
{
$\begin{tz}
  \draw(0,0) node[below] {$A$} to (0,1) node[morphlabel] {$f_1$} to node[right] {$B$} (0,2.5) node[morphlabel] {$g$} to node[right] {$C$} (0,4) node[morphlabel] {$f_3$} to (0,5) node[above] {$D$};
\end{tz}$
};  

\begin{scope}[double arrow scope]
	\draw (1) -- node[above] {$\alpha$} (2);
\end{scope}
\end{tz}
\]
stands for the 2-morphism $\id_{f_3} \ast \alpha \ast \id_{f_1}$. As is well known, the usual `horizontal composition' of 2-morphisms in a 2-category can be described solely in terms of `vertical composition', and whiskering. So, if $f_1, g_1 \colon A \rightarrow B$ and $f_2, g_2 \colon B \rightarrow C$ are 1-morphisms, and $\alpha \colon f_1 \Rightarrow g_1$ and $\beta \colon f_2 \Rightarrow g_2$ are 2-morphisms, then we have:
\begin{equation} \label{whisker_1}
 \beta \ast \alpha = (\beta \ast \id_{f_1}) \circ (\id_{f_2} \ast \alpha) = (\id_{f_2} \ast \alpha) \circ (\beta \ast \id_{f_1})
\end{equation}
We can view this equation as {\em defining} horizontal composition. In wire diagrams, we will draw $\beta \ast \alpha$ as if it is being applied simultaneously. So, the equations \eqref{whisker_1} look as follows in wire diagrams:
\begin{eqnarray*} \tikzset{every picture/.style={yscale=0.9}}
\begin{tz}[xscale=1.5] 
\node (1) at (0,0)
{
$\begin{tz}
  \draw(0,0) node[below] {$A$} to (0,1) node[morphlabel] (A) {$f_1$} to node[right] {$B$} (0,2.5) node[morphlabel] (B) {$f_2$} to (0,3.5) node[above] {$C$};
  \selectpart[green, inner sep=2pt]{(A)};
  \selectpart[blue, inner sep=2pt]{(B)};
\end{tz}$
};  

\node (2) at (1,0)
{
$\begin{tz}
  \draw(0,0) node[below] {$A$} to (0,1) node[morphlabel] (A) {$g_1$} to node[right] {$B$} (0,2.5) node[morphlabel] (B) {$g_2$} to (0,3.5) node[above] {$C$};
\end{tz}$
};  

\begin{scope}[double arrow scope]
	\draw (1) -- node[above, blue] {$\beta$} node[below, green] {$\alpha$} (2);
\end{scope}
\end{tz}
& := & \tikzset{every picture/.style={yscale=0.9}}
\begin{tz}[xscale=1.5] 
\node (1) at (0,0)
{
$\begin{tz}
  \draw(0,0) node[below] {$A$} to (0,1) node[morphlabel] (A) {$f_1$} to node[right] {$B$} (0,2.5) node[morphlabel] (B) {$f_2$} to (0,3.5) node[above] {$C$};
  \selectpart[green, inner sep=2pt]{(A)};
\end{tz}$
};  

\node (2) at (1,0)
{
$\begin{tz}
  \draw(0,0) node[below] {$A$} to (0,1) node[morphlabel] (A) {$g_1$} to node[right] {$B$} (0,2.5) node[morphlabel] (B) {$g_2$} to (0,3.5) node[above] {$C$};
  \selectpart[blue, inner sep=2pt]{(B)};
\end{tz}$
};  

\node (3) at (2,0)
{
$\begin{tz}
  \draw(0,0) node[below] {$A$} to (0,1) node[morphlabel] (A) {$g_1$} to node[right] {$B$} (0,2.5) node[morphlabel] (B) {$g_2$} to (0,3.5) node[above] {$C$};
\end{tz}$
};

\begin{scope}[double arrow scope]
	\draw (1) -- node[above, green] {$\alpha$} (2);
	\draw (2) -- node[above, blue] {$\beta$} (3);
\end{scope}
\end{tz} \\
&=& \begin{tz}[xscale=1.5]  \tikzset{every picture/.style={yscale=0.9}}
\node (1) at (0,0)
{
$\begin{tz}
  \draw(0,0) node[below] {$A$} to (0,1) node[morphlabel] (A) {$f_1$} to node[right] {$B$} (0,2.5) node[morphlabel] (B) {$f_2$} to (0,3.5) node[above] {$C$};
  \selectpart[blue, inner sep=2pt]{(B)};
\end{tz}$
};  

\node (2) at (1,0)
{
$\begin{tz}
  \draw(0,0) node[below] {$A$} to (0,1) node[morphlabel] (A) {$g_1$} to node[right] {$B$} (0,2.5) node[morphlabel] (B) {$g_2$} to (0,3.5) node[above] {$C$};
  \selectpart[green, inner sep=2pt]{(A)};
\end{tz}$
};  

\node (3) at (2,0)
{
$\begin{tz}
  \draw(0,0) node[below] {$A$} to (0,1) node[morphlabel] (A) {$g_1$} to node[right] {$B$} (0,2.5) node[morphlabel] (B) {$g_2$} to (0,3.5) node[above] {$C$};
\end{tz}$
};  

\begin{scope}[double arrow scope]
	\draw (1) -- node[above, blue] {$\beta$} (2);
	\draw (2) -- node[above, green] {$\alpha$} (3);
\end{scope}
\end{tz}
\end{eqnarray*}

\section{Semistrict monoidal 2-categories} \label{semistrict_sec}
In this section we recall the notion of a semistrict monoidal 2-category~\cite{BN96, ds97-mbh, kv94-2categories, kv94-bm2} in the formulation of Crans~\cite{cr98-gcb}. 

The category $\text{2Cat}$ of strict 2-categories and strict 2-functors can be equipped with the {\em Gray tensor product} $\otimes_G$ making it into a monoidal category~\cite{GPS95} (see~\cite{gurskithesis} for an exposition). The most important feature of the Gray tensor product $\bicat{C} \otimes_G \bicat{D}$ of two strict 2-categories is that the objects of $\bicat{C} \otimes_G \bicat{D}$ are the same as the objects of $\bicat{C} \times \bicat{D}$, and that for every pair of 1-morphisms $f \colon A \rightarrow A'$ in $\bicat{C}$ and $g \colon B \rightarrow B'$ in $\bicat{D}$ there is a 2-isomorphism\footnote{Note that our convention runs counter to that of Baez and Neuchl~\cite{BN96}, but when thought of as a cubical functor fits the standard definition of pseudofunctor~\cite{L98} correctly.}
\[
	\begin{tz}[scale=3]
		\node (1)  at (0,0) {$(A,B)$};
		\node (3) at (0,-1) {$(A,B')$};
		\node (2) at (1,0) {$(A',B)$};
		\node (4) at (1,-1) {$(A',B')$};
		\draw[->] (1) to node[above] {$f \otimes_G \id_B$} (2);
		\draw[->] (1) to node[left] {$\id_A \otimes_G g$} (3);
		\draw[->] (3) to node[below] {$f \otimes_G \id_{B'}$} (4);
		\draw[->] (2) to node[right] {$\id_{A'} \otimes_G g$} (4);
\begin{scope}[double arrow scope]
		\draw (0.2, -0.8) to node[below right] {$\gamma_{f,g}$} (0.8, -0.2);
\end{scope}
	\end{tz}
\]
in $\bicat{C} \otimes_G \bicat{D}$. A {\em semistrict monoidal 2-category} is then usually defined as a monoid in the monoidal category $(\text{2Cat}, \otimes_G)$. 

In this way, Gray categories are used as a technical construct to avoid leaving the world of strict 2-categories and strict 2-functors. However, the explicit algebraic definition of the Gray tensor product $\bicat{C} \otimes_G \bicat{D}$ is rather awkward, given by a long list of generators and relations~\cite[Section 5.1]{gurskithesis}. In practice, the notion of a {\em cubical functor} is used instead. 

\begin{defn} Suppose $\bicat{C}$, $\bicat{C'}$ and $\bicat{D}$ are strict 2-categories. A {\em cubical functor} $F \colon \bicat{C} \times \bicat{C'} \rightarrow \bicat{D}$ is a pseudofunctor whose coherence isomorphisms
\begin{equation} \label{pseudo}
 \Phi_{(f',g'),(f,g)} \colon F(f',g') \circ F(f,g) \Rightarrow F(f'f, g'g)
\end{equation}
are the identity 2-morphism if $f' = \id$ or if $g=\id$.
\end{defn}

Note that we have not listed unit coherence 2-isomorphisms $u_{(A,B)} \colon \id_{F(A,B)} \Rightarrow F(\id_A, \id_B)$ as part of the data of a cubical functor, since it follows from (a) the fact that all the 2-categories are strict, (b) the cubical condition, and (c) the unit equation on $u_{(A,B)}$ in a pseudofunctor, that each $u_{(A,B)}$ must be the identity. 

\begin{proposition}\textnormal{(~\cite{Gr76},~\cite{Gr74},~\cite[Thm 5.2.5]{gurskithesis})} There is a canonical isomorphism  
\[
\mathrm{Cub} (\bicat{C} \times \bicat{C'}, \, \bicat{D}) \cong \mathrm{2Cat}(\bicat{C} \otimes_G \bicat{C'}, \, \bicat{D}) 
\]
between the set of cubical functors from $\bicat{C} \times \bicat{C'}$ to $\bicat{D}$ and the set of strict 2-functors from $\bicat{C} \otimes_G \bicat{C'}$ to $\bicat{D}$.
\end{proposition}
Let us write $\text{2Cat}^{\text{ps}}$ for the category whose objects are strict 2-categories and whose morphisms are pseudofunctors. It forms a monoidal category $(\text{2Cat}^{\text{ps}}, \, \times)$ under Cartesian product of 2-categories. With the above discussion in mind, the following definition is normally used in practice (if not explicitly so then implicitly so!).

\begin{defn} \label{strict2} A {\em semistrict monoidal 2-category} is a monoid $(\bicat{M}, \, 1, \, \otimes, \, \{\Phi_{(f',g'),(f,g)}\})$ in the monoidal category $(\text{2Cat}^{\text{ps}}, \, \times)$ whose tensor product pseudofunctor 
\[
(\otimes, \Phi) \colon \bicat{M} \times \bicat{M} \rightarrow \bicat{M}
\]
is cubical. \label{strict1}
\end{defn}
An alternative, possibly more natural, way to define a semistrict monoidal 2-category is to start with the definition of a fully weak monoidal bicategory~\cite{st13-ccb, CSPthesisLatest} and then impose strictness conditions on the coherence data.
\begin{defn}  A {\em semistrict monoidal 2-category} is a monoidal bicategory $\bicat{M}$ such that:
\begin{itemize}
\item $\bicat{M}$ is a strict 2-category;
\item The transformations $\alpha$, $r$, $\pi$, $\mu$, $\lambda$ and $\rho$ are identities. Moreover the inverse adjoint equivalences $\alpha^*$, $l^*$ and $r^*$ are also identities with trivial adjunction data. 
\item The functor $\otimes = (\otimes, \Phi_{(f',g'),(f,g)}, \Phi_{A, B})$ is cubical, and $\Phi_{A,B}$ is the identity for all objects $A,B$.  
\end{itemize}
\end{defn}
\noindent If we unravel the many diagrams defining a monoidal bicategory from~\cite{st13-ccb, CSPthesisLatest}, and impose the above equations, we conclude that these two definitions are identical. 

In fact, this definition contains redundadnt information. 

\begin{defn} Let $\bicat{M}$ be a semistrict monoidal 2-category. The underlying {\em interchangor} is the collection of 2-isomorphisms
\begin{equation} \label{interchange}
 \phi_{f,g} := \Phi_{(f, \id), (\id, g)} \colon (f \otimes \id) \circ (\id \otimes g) \Rightarrow (\id \otimes g) \circ (f \otimes \id)
\end{equation}
where $f,g$ are 1-morphisms in $\bicat{M}$. 
\end{defn}
\noindent Note that the target of $\Phi_{(f, \id), (\id, g)}$ in \eqref{interchange} makes sense, since if we unravel the definitions, we obtain:
\[
\begin{tz}[xscale=6, yscale=3]
 \node (1) at (0,0) {$(f \otimes \id) \circ (\id \otimes g)$};
 \node(2) at (1,0) {$(f \circ \id) \otimes (\id \circ g)$};
 \node (3) at (0,-1) {$(\id \otimes g) \circ (f \otimes \id)$};
 \node (4) at (1,-1) {$(\id \circ f) \otimes (g \circ \id)$};

 \draw[double] (2) to (4);

\begin{scope}{double arrow scope}
  \draw[double, ->] (1) to node[above] {$\Phi_{(f, \id), (\id, g)}$} (2);
  \draw[double, ->] (4) to node[below] {$\Phi^{-1}_{(\id, g), (f, \id)} = \id$} (3);
\end{scope}

\end{tz}
\]
The following lemma is standard. We will give a graphical proof in terms of wire diagrams in part 3 of Proposition \ref{stringentprop} below.
\begin{lemma}\textnormal{(~\cite{CSPthesisLatest},~\cite[Lemma 2.15]{bms13-gdd})} The coherence 2-isomorphisms $\Phi_{(f', g'),(f,g)}$ are uniquely determined by the underlying interchangor 2-isomorphisms $\phi_{f,g}$. \label{little_lem}
\end{lemma} 

\section{Stringent monoidal 2-categories} \label{stringent_sec}
In this section we introduce stringent monoidal 2-categories, extend the wire diagram notation to them, and prove that they are equivalent to semistrict monoidal 2-categories. 

\subsection{The definition}

In the light of Lemma \ref{little_lem}, it is convenient to formulate the notion of a semistrict monoidal 2-category purely in terms of the interchangor 2-isomorphisms. This has been the approach in~\cite{BN96,bms13-gdd}. We make the following definition, apologizing to the reader for the burden of excessive terminology, in the hope that it is compensated for by the boon of greater precision.
\begin{defn} A {\em stringent} monoidal 2-category $\bicat{M} =(\bicat{M}, 1, \otimes, \{\phi_{f,g}\})$ consists of:
\begin{itemize}
 \item a strict 2-category $\bicat{M}$, 
 \item an object $1 \in \bicat{M}$, 
 \item strict left- and right-tensor functors $A \otimes -$ and $- \otimes A$ from $\bicat{M}$ to itself, 
 \item interchangor 2-isomorphisms $\phi_{f,g} \colon  (f \otimes \id) \circ (\id \otimes g) \Rightarrow (\id \otimes g) \circ (f \otimes \id)$ 
\end{itemize}
satisfying the relations spelt out below.
\end{defn}
\noindent Instead of using {\em pasting diagrams} to describe these relations, as in~\cite[Lemma 4]{BN96}, I will extend the {\em wire diagrams} notation from Section \ref{wire_diag}. This extended notation will be introduced as we go along. Let us begin. 

So, to start with, a stringent monoidal 2-category consists of a strict 2-category $\bicat{M}$ together with:
\begin{itemize}
 \item[(i)] An object $1 \in \bicat{M}$, drawn as the invisible wire:
 \[
   \begin{tz}
    \draw[dotted] (0,0) rectangle (1,1);
   \end{tz}
 \]
 \item[(ii)] For any two objects $A, B \in \bicat{M}$, an object $A \otimes B \in \bicat{M}$, drawn as:
  \[
  \begin{tz} 
  \draw(0,0) node[below] {$A$} to (0,1) node[above] {$A$}; 
  \draw(1,0) node[below] {$B$} to (1,1) node[above] {$B$}; 
  \end{tz} 
 \]
 \item[(iii)] The tensor product of objects is strictly associative and unital. Moreover, for each object $A \in \bicat{M}$, it extends to a strict 2-functor $L_A := A \otimes -$ and $R_A := - \otimes A$. These 2-functors satisfy $L_A  L_B = L_{A \otimes B}$, $R_B R_A = R_{A \otimes B}$ and $L_{A} R_B = R_B L_A$ for all $A, B \in \bicat{M}$.
\end{itemize}

Let us pause here to explain the wire diagram notation more precisely. Each wire diagram representing a 1-morphism is to be evaluated into a 1-morphism in $\bicat{M}$ according to the prescription {\em tensor first, then compose}. For instance, the diagram
\begin{equation} \label{d1}
\begin{tz}
  \draw(0,0) node[below] {$A$} to (0,1) node[morphlabel] (f) {$f$} to node[right] {$A'$} (0,2.5) node[morphlabel] (f') {$f'$} to (0, 5) node[above] {$A''$};
  \draw (1,0) node[below] {$B$} to (1,1) node[morphlabel] (g) {$g$} to node[right] {$B'$} (1,4) node[morphlabel] (g') {$g'$} to (1,5) node[above] {$B''$};
\end{tz}
\end{equation}
is to be evaluated as follows. First, draw horizontal lines to separate the diagram into its indecomposable pieces. The regions between the horizontal lines evaluate to tensor products of objects, and the horizontal lines evaluate to tensor products of 1-morphisms. Then, compose the 1-morphisms together:
\[
\begin{tz}[yscale=1.1]
  \draw(0,0) node[below] {$A$} to (0,1) node[morphlabel] (f) {$f$} to node[right] {$A'$} (0,2.5) node[morphlabel] (f') {$f'$} to (0, 5) node[above] {$A''$};
  \draw (1,0) node[below] {$B$} to (1,1) node[morphlabel] (g) {$g$} to node[right] {$B'$} (1,4) node[morphlabel] (g') {$g'$} to (1,5) node[above] {$B''$};
  \draw[dotted] (0,1) to (2,1);
  \draw[dotted] (0,2.5) to (2,2.5);
  \draw[dotted] (0,4) to (2, 4);
\node[below] (1) at (3.5,0) {$A \otimes B$};
\node (2) at (3.5,1.75) {$A' \otimes B'$};
\node (3) at (3.5,3.25) {$A'' \otimes B'$};
\node[above] (4) at (3.5,5) {$A'' \otimes B''$};

\draw[->] (1) to (2);
\draw[->] (2) to (3);
\draw[->] (3) to (4);

\node at (2.8, 1) {$f \otimes g$};
\node at (2.8, 2.5) {$f' \otimes \id_{B'}$};
\node at (2.8, 4) {$\id_{A'} \otimes g'$};
  
\end{tz}
\]
So, the wire diagram \eqref{d1} evaluates to the 1-morphism
\[
(\id_{A''} \otimes g') \circ (f' \otimes \id_{B'}) \circ (f \otimes g).
\]
in $\bicat{M}$. 

With this prescription, we can interpret (iii) as follows. Functoriality at the level of 1-morphisms means that equations between composites of 1-morphisms are {\em local}, that is they remain true after arbitrarily tensoring on the left and right and pre- and post-composition:
\[
\begin{tz}[xscale=0.5]
	\draw (0,0) to (0,4);
	\draw (2,0) to (2,4);
	\draw[fill=white] (-0.2, 1.03) rectangle (2.2, 1.63);
	\draw[fill=white] (-0.2, 2.36) rectangle (2.2, 2.96);
	\node at (1, 1.33) {$f$};
	\node at (1, 2.66) {$g$};
	\node at (1, 0.2) {$\cdots$};
	\node at (1,3.8) {$\cdots$};
	\node at (1, 2) {$\cdots$};
\end{tz}
\,\, = \, \,
\begin{tz}[xscale=0.5]
	\draw (0,0) to (0,4);
	\draw (2,0) to (2,4);
	\draw[fill=white] (-0.2, 1.6) rectangle (2.2, 2.4);
	\node at (1, 2) {$h$};
	\node at (1, 0.2) {$\cdots$};
	\node at (1,3.8) {$\cdots$};
\end{tz}
 \quad \Rightarrow \quad
 \begin{tz}[xscale=0.5]
	\draw (0,0) to (0,4);
	\draw (2,0) to (2,4);
	\draw[fill=white] (-0.2, 1.03) rectangle (2.2, 1.63);
	\draw[fill=white] (-0.2, 2.36) rectangle (2.2, 2.96);
	\node at (1, 1.33) {$f$};
	\node at (1, 2.66) {$g$};
	\node at (1, 0.2) {$\cdots$};
	\node at (1,3.8) {$\cdots$};
	\node at (1, 2) {$\cdots$};
	\draw[draw=green] (-0.5, 0) rectangle (2.5, 4);
	\node at (1, 4.5) {$\vdots$};
	\node at (3, 2) {$\cdots$};
	\node at (-1, 2) {$\cdots$};
	\node at (1, -0.5) {$\vdots$};
\end{tz}
\,\, = \,\,
 \begin{tz}[xscale=0.5]
	\draw (0,0) to (0,4);
	\draw (2,0) to (2,4);
	\draw[fill=white] (-0.2, 1.6) rectangle (2.2, 2.4);
	\node at (1, 2) {$h$};
	\node at (1, 0.2) {$\cdots$};
	\node at (1,3.8) {$\cdots$};
	\draw[draw=green] (-0.5, 0) rectangle (2.5, 4);
	\node at (1, 4.5) {$\vdots$};
	\node at (3, 2) {$\cdots$};
	\node at (-1, 2) {$\cdots$};
	\node at (1, -0.5) {$\vdots$};
\end{tz}
\]
Functoriality at the level of 2-morphisms is similar: if an equation between composites of 2-morphisms holds, then it continues to hold after arbitrarily tensoring the left and right hand sides and pre- and post-composing with 1-morphisms.
 
If a 2-morphism $\alpha \colon f \Rightarrow g$ is surrounded by tensor products and composites of 1-morphisms, then we use a box to indicate where $\alpha$ is acting. So for instance, 
\[
\begin{tz}[xscale=3.3]

\node (1) at (0,0)
{
$\begin{tz}[xscale=0.8]
	\draw(0,0) node[below] {$A$} to node[morphlabel] {$h_3$} (0,5) node[above] {$A'$};
	\draw(1,0) node[below] {$B$} to (1,1) node[morphlabel] {$h_4$} to node[right] {$B'$} (1, 2.5) node[morphlabel] (f) {$f$} to node[right] {$B''$} (1, 4) node[morphlabel]{$h_2$} to (1,5) node[above] {$B'''$};
	\draw(2,0) node[below] {$C$} to node[morphlabel] {$h_1$} (2,5) node[above] {$C'$};
	\selectpart[green, inner sep=1pt] {(f)};
\end{tz}$
};

\node (2) at (1,0)
{
$\begin{tz}[xscale=0.8]
	\draw(0,0) node[below] {$A$} to node[morphlabel] {$h_3$} (0,5) node[above] {$A'$};
	\draw(1,0) node[below] {$B$} to (1,1) node[morphlabel] {$h_4$} to node[right] {$B'$} (1, 2.5) node[morphlabel] (f) {$g$} to node[right] {$B''$} (1, 4) node[morphlabel]{$h_2$} to (1,5) node[above] {$B'''$};
	\draw(2,0) node[below] {$C$} to node[morphlabel] {$h_1$} (2,5) node[above] {$C'$};
\end{tz}$
};

\begin{scope}[double arrow scope]
	\draw (1) -- node[above] {$\alpha$} (2);
\end{scope}
\end{tz}
\]
evaluates in gory detail as
\[
\id_{\id_{A'} \otimes h_2 \otimes \id_{C'}} \ast (\id_{h_3} \otimes \alpha \otimes \id_{h_1}) \ast \id_{\id_{A} \otimes h_4 \otimes \id_{C'}}.
\]
At this point the utility of the wire diagrams notation starts to become clear!  

\begin{itemize} 
\item[(iv)] For every pair of 1-morphisms $f \colon A \rightarrow A'$ and $g \colon B \rightarrow B'$, an {\em interchangor} 2-isomorphism 
\[
\phi_{f,g} \colon (f \otimes \id_{B'}) \circ (\id_{A} \otimes g) \Rightarrow (\id_{A'} \otimes g) \circ (f \otimes \id_{B})
\]
drawn as:
\begin{equation} \label{interchangor}
\begin{tz}[xscale=1.3] 
\node (1) at (0,0)
{
$\begin{tz}
  \draw(0,0) node[below] {$A$} to (0,2) node[morphlabel] (f) {$f$} to (0,3) node[above] {$A'$};
  \draw(1,0) node[below] {$B$} to (1,1) node[morphlabel] (g) {$g$} to (1,3) node[above] {$B'$};
\end{tz}$
};  

\node (2) at (2,0)
{
$\begin{tz}
  \draw(0,0) node[below] {$A$} to (0,1) node[morphlabel] (f) {$f$} to (0,3) node[above] {$A'$};
  \draw(1,0) node[below] {$B$} to (1,2) node[morphlabel] (g) {$g$} to (1,3) node[above] {$B'$};
\end{tz}$
};  

\begin{scope}[double arrow scope]
	\draw (1) -- node[above] {$\phi_{f,g}$} (2);
\end{scope}

\end{tz}
\end{equation}
\end{itemize}
We pause here to unpack a crucial identity. If $\phi$ is the underlying interchangor of the coherence isomorphisms $\Phi$ in a semistrict monoidal 2-category, then the cubical equation on $\Phi$ implies the following in wire diagrams:
\[
\begin{tz}
  \draw(0,0) node[below] {$A$} to (0,1) node[morphlabel] {$f$} to (0,3.5) node[above] {$A'$};
  \draw(1,0) node[below] {$B$} to (1,2.5) node[morphlabel] {$g$} to (1,3.5) node[above] {$B'$};
\end{tz}
\quad = \quad
\begin{tz}
  \draw(0,0) node[below] {$A$} to (0,1) node[morphlabel] {$f$} to (0,2.5) node[morphlabel] {$\id$} to (0,3.5) node[above] {$A'$};
  \draw(1,0) node[below] {$B$} to (1,1) node[morphlabel] {$\id$} to (1,2.5) node[morphlabel] {$g$} to (1, 3.5) node[above] {$B'$};
\end{tz}
\quad = \quad
\begin{tz}
  \draw(0,0) node[below] {$A$} to (0,1.75) node[morphlabel] {$\id \circ f$} to (0,3.5) node[above] {$A'$};
  \draw(1,0) node[below] {$B$} to (1,1.75) node[morphlabel] {$g \circ \id$} to (1,3.5) node[above] {$B'$};
\end{tz}
\quad = \quad
\begin{tz}
  \draw(0,0) node[below] {$A$} to (0,1.75) node[morphlabel] {$f$} to (0,3.5) node[above] {$A'$};
  \draw(1,0) node[below] {$B$} to (1,1.75) node[morphlabel] {$g$} to (1,3.5) node[above] {$B'$};
\end{tz}
\]
The first equation follows from left-tensoring and right-tensoring being strict 2-functors. The third equation follows since $\bicat{M}$ is a strict 2-category. The second equation follows from $\Phi_{(\id, g), (f, \id)} = \id$. To emphasize: although the interchangor \eqref{interchangor} is nontrivial, we at least have the following identity, which we take as an {\em axiom} in a stringent monoidal 2-category.
\begin{itemize}
\item [(v)] For all 1-morphisms $f\colon A \rightarrow A'$ and $g \colon B \rightarrow B'$, we have:
\begin{center}
\framebox[1.1\width]{$
\begin{tz}
  \draw(0,0) node[below] {$A$} to (0,1) node[morphlabel] {$f$} to (0,3) node[above] {$A'$};
  \draw(1,0) node[below] {$B$} to (1,2) node[morphlabel] {$g$} to (1,3) node[above] {$B'$};
\end{tz}
\quad = \quad
\begin{tz}
  \draw(0,0) node[below] {$A$} to (0,1.5) node[morphlabel] {$f$} to (0,3) node[above] {$A'$};
  \draw(1,0) node[below] {$B$} to (1,1.5) node[morphlabel] {$g$} to (1,3) node[above] {$B'$};
\end{tz} \, .$}
\end{center}
\end{itemize}
This is called {\em nudging} in~\cite{GPS95}. Note that if we had adopted the `opcubical' convention on cubical functors as in~\cite{bms13-gdd}, nudging would have worked in the opposite direction. 

\begin{itemize}
\item[(vi)] For all 1-morphisms $f \colon A \rightarrow A'$, $g \colon B \rightarrow B'$, $h \colon C \rightarrow C'$ we have
\[
\phi_{\id_A \otimes g, h} = \id_A \otimes \phi_{g,h}, \quad \phi_{f \otimes \id_B, h} = \phi_{f, \id_B \otimes h}, \, \text{and } \phi_{f, g \otimes \id_C} = \phi_{f,g} \otimes \id_C.
\]
\end{itemize}
The first equation says that
\[
\begin{tz}[xscale=3.3] 
\node (1) at (0,0)
{
$\begin{tz}[xscale=0.6]
  \draw (-1, 0) node[below] {$A$} to (-1, 3) node[above] {$A$};
  \draw(0,0) node[below] {$B$} to (0,2) node[morphlabel] {$g$} to (0,3) node[above] {$B'$};
  \draw(1,0) node[below] {$C$} to (1,1) node[morphlabel] {$h$} to (1,3) node[above] {$C'$};
  \draw[draw=green] (-1.4, 2.4) rectangle (1.4, 0.6);
\end{tz}$
};  

\node (2) at (1,0)
{
$\begin{tz}[xscale=0.6]
  \draw (-1, 0) node[below] {$A$} to (-1, 3) node[above] {$A$};
  \draw(0,0) node[below] {$B$} to (0,1) node[morphlabel] {$g$} to (0,3) node[above] {$B'$};
  \draw(1,0) node[below] {$C$} to (1,2) node[morphlabel] {$h$} to (1,3) node[above] {$C'$};
\end{tz}$
};  

\begin{scope}[double arrow scope]
	\draw (1) -- node[above] {$\phi_{\id_A \otimes g, h}$} (2);
\end{scope}
\end{tz} 
\quad = \quad
\begin{tz}[xscale=3]
\node (1) at (0,0)
{
$\begin{tz}[xscale=0.6]
  \draw (-1, 0) node[below] {$A$} to (-1, 3) node[above] {$A$};
  \draw(0,0) node[below] {$B$} to (0,2) node[morphlabel] {$g$} to (0,3) node[above] {$B'$};
  \draw(1,0) node[below] {$C$} to (1,1) node[morphlabel] {$h$} to (1,3) node[above] {$C'$};
  \draw[draw=green] (-0.4, 2.4) rectangle (1.4, 0.6);
\end{tz}$
};  

\node (2) at (1,0)
{
$\begin{tz}[xscale=0.6]
  \draw (-1, 0) node[below] {$A$} to (-1, 3) node[above] {$A$};
  \draw(0,0) node[below] {$B$} to (0,1) node[morphlabel] {$g$} to (0,3) node[above] {$B'$};
  \draw(1,0) node[below] {$C$} to (1,2) node[morphlabel] {$h$} to (1,3) node[above] {$C'$};
\end{tz}$
};

\begin{scope}[double arrow scope]
	\draw (1) -- node[above] {$\phi_{g, h}$} (2);
\end{scope}
\end{tz} 
\]
and similarly for the other two equations.
\begin{itemize}
 \item [(vii)] For all 1-morphisms $f \colon A \rightarrow A'$ and $g \colon B \rightarrow B'$ we have $\phi_{f, \id} = \id$ and $\phi_{\id, g} = \id$. In diagrams:
\[
\begin{tz}[xscale=1.3] 
\node (1) at (0,0)
{
$\begin{tz}
  \draw(0,0) node[below] {$A$} to (0,2) node[morphlabel] (f) {$f$} to (0,3) node[above] {$A'$};
  \draw(1,0) node[below] {$B$} to (1,1) node[morphlabel] (g) {$\id$} to (1,3) node[above] {$B'$};
\end{tz}$
};  

\node (2) at (2,0)
{
$\begin{tz}
  \draw(0,0) node[below] {$A$} to (0,1) node[morphlabel] (f) {$f$} to (0,3) node[above] {$A'$};
  \draw(1,0) node[below] {$B$} to (1,2) node[morphlabel] (g) {$\id$} to (1,3) node[above] {$B'$};
\end{tz}$
};  

\begin{scope}[double arrow scope]
	\draw (1) -- node[above] {$\phi_{f,\id_B}$} (2);
\end{scope}

\end{tz}
\, \, =  \,\,
\begin{tz}[xscale=1.3] 
\node (1) at (0,0)
{
$\begin{tz}
  \draw(0,0) node[below] {$A$} to (0,2) node[morphlabel] (f) {$f$} to (0,3) node[above] {$A'$};
  \draw(1,0) node[below] {$B$} to (1,1) node[morphlabel] (g) {$\id$} to (1,3) node[above] {$B'$};
\end{tz}$
};  

\node (2) at (2,0)
{
$\begin{tz}
  \draw(0,0) node[below] {$A$} to (0,1) node[morphlabel] (f) {$f$} to (0,3) node[above] {$A'$};
  \draw(1,0) node[below] {$B$} to (1,2) node[morphlabel] (g) {$\id$} to (1,3) node[above] {$B'$};
\end{tz}$
};  

\begin{scope}[double arrow scope]
	\draw (1) -- node[above] {$\id$} (2);
\end{scope}

\end{tz}
\]
plus the other version of this equation (where $f$ occurs on the right).
 \item [(viii)] For all 1-morphisms $f \colon A \rightarrow A'$, $g \colon B \rightarrow B'$ and 2-morphisms $\alpha \colon f \Rightarrow f'$, the following equation holds:
\[
\begin{tz}[xscale=1.3, yscale=6] 
\node (1) at (0,0)
{
$\begin{tz}
  \draw(0,0) node[below] {$A$} to (0,2) node[morphlabel] (f) {$f$} to (0,3) node[above] {$A'$};
  \draw(1,0) node[below] {$B$} to (1,1) node[morphlabel] (g) {$g$} to (1,3) node[above] {$B'$};
  \selectpart[green, inner sep=1pt] {(f)};
\end{tz}$
};  

\node (2) at (2,0)
{
$\begin{tz}
  \draw(0,0) node[below] {$A$} to (0,2) node[morphlabel] (f) {$f'$} to (0,3) node[above] {$A'$};
  \draw(1,0) node[below] {$B$} to (1,1) node[morphlabel] (g) {$g$} to (1,3) node[above] {$B'$};
\end{tz}$
};  

\node (3) at (0,-1)
{
$\begin{tz}
  \draw(0,0) node[below] {$A$} to (0,1) node[morphlabel] (f) {$f$} to (0,3) node[above] {$A'$};
  \draw(1,0) node[below] {$B$} to (1,2) node[morphlabel] (g) {$g$} to (1,3) node[above] {$B'$};
  \selectpart[inner sep=1pt, green] {(f)}
\end{tz}$
};

\node (4) at (2,-1)
{
$\begin{tz}
  \draw(0,0) node[below] {$A$} to (0,1) node[morphlabel] (f) {$f'$} to (0,3) node[above] {$A'$};
  \draw(1,0) node[below] {$B$} to (1,2) node[morphlabel] (g) {$g$} to (1,3) node[above] {$B'$};
\end{tz}$
};

\begin{scope}[double arrow scope]
	\draw (1) -- node[above, green] {$\alpha$} (2);
	\draw (2) -- node[right] {$\phi_{f, g'}$} (4);
	\draw (1) -- node[left] {$\phi_{f,g}$} (3);
	\draw (3) -- node[below, green] {$\alpha$} (4);
\end{scope}
\end{tz}
\] 
Similarly for a 2-morphism $\beta \colon g \Rightarrow g'$.
\item[(ix)] For all 1-morphisms $f \colon A \rightarrow A'$, $g \colon B \rightarrow B'$, $h \colon B' \rightarrow B''$, the following diagram commutes:
\[
\begin{tz}[xscale=1.3, yscale=8] 
\node (1) at (0,0)
{
$\begin{tz}[yscale=0.7]
  \draw(0,0) node[below] {$A$} to (0,4) node[morphlabel] (f) {$f$} to (0,5) node[above] {$A'$};
  \draw(1,0) node[below] {$B$} to (1,1) node[morphlabel] (g) {$g$} to (1,2.5) node[morphlabel] (g') {$g'$} to (1, 5) node[above] {$B''$};
	\draw[green] ([shift={(135:0.1)}]f.135) rectangle ([shift={(-45:0.1)}] g'.-45);
	\draw[blue] ([shift={(135:0.3)}]f.135) rectangle ([shift={(-45:0.3)}] g.-45);
\end{tz}$
};  

\node (2) at (2,0)
{
$\begin{tz}[yscale=0.7]
  \draw(0,0) node[below] {$A$} to (0,2.5) node[morphlabel] (f) {$f$} to (0,5) node[above] {$A'$};
  \draw(1,0) node[below] {$B$} to (1,1) node[morphlabel] (g) {$g$} to (1,4) node[morphlabel] (g') {$g'$} to (1, 5) node[above] {$B''$};
	\draw[green] ([shift={(135:0.1)}]f.135) rectangle ([shift={(-45:0.1)}] g.-45);
\end{tz}$    
};  

\node (3) at (2,-1)
{
$\begin{tz}[yscale=0.7]
  \draw(0,0) node[below] {$A$} to (0,1) node[morphlabel] (f) {$f$} to (0,5) node[above] {$A'$};
  \draw(1,0) node[below] {$B$} to (1,2.5) node[morphlabel] (g) {$g$} to (1,4) node[morphlabel] (g') {$g'$} to (1, 5) node[above] {$B''$};
\end{tz}$
};

\begin{scope}[double arrow scope]
	\draw (1) -- node[above, green] {$\phi_{f,g'}$} (2);
	\draw (2) -- node[right] {$\phi_{f, g}$} (3);
	\draw (1) -- node[below left, blue] {$\phi_{f,g'g}$} (3);
\end{scope}
\end{tz}
\]  
Note that we have used colours to differentiate the source of the 2-morphisms. There is also a corresponding rotated version of this diagram.
\end{itemize}

\subsection{Example} A stringent monoidal 2-category with one object $\bicat{M}$ is the same thing as a symmetric monoidal category $M$, after reindexing the 1-morphisms and 2-morphisms in $\bicat{M}$ as objects and morphisms in $M$ respectively. Wire diagrams make this very clear. Let $A$ and $B$ be 1-endomorphisms of the unit object $1 \in \bicat{M}$. The tensor product $\odot$ in $M$ is defined as $A \odot B := A \otimes B$ in $\bicat{M}$. The braiding 
\[
\sigma_{A,B} \colon A \odot B \rightarrow B \odot A
\]
in $M$ is defined by using the interchangor $\phi_{B,A}$ in $\bicat{M}$, as follows:
\[
\begin{tz}[yscale=0.6, xscale=0.8]
  \draw [dotted] (0,0) to (0,1.5) node[morphlabel] {$A$} to (0,3);
  \draw [dotted] (1,0) to (1,1.5) node[morphlabel] {$B$} to (1,3);
\end{tz}
\quad = \quad
\begin{tz}[yscale=0.6, xscale=0.8]
  \draw [dotted] (0,0) to (0,1) node[morphlabel] {$A$} to (0,3);
  \draw [dotted] (1,0) to (1,2) node[morphlabel] {$B$} to (1,3);
\end{tz}
\quad = \quad
\begin{tz}[yscale=0.6, xscale=0.8]
  \draw [dotted] (0,0) to (0,2) node[morphlabel] {$B$} to (0,3);
  \draw [dotted] (1,0) to (1,1) node[morphlabel] {$A$} to (1,3);
\end{tz}
 \,\, \xdoubleto{\phi_{B,A}} \,\,
\begin{tz}[yscale=0.6, xscale=0.8]
  \draw [dotted] (0,0) to (0,1) node[morphlabel] {$B$} to (0,3);
  \draw [dotted] (1,0) to (1,2) node[morphlabel] {$A$} to (1,3);
\end{tz}
\quad = \quad
\begin{tz}[yscale=0.6, xscale=0.8]
  \draw [dotted] (0,0) to (0,1.5) node[morphlabel] {$B$} to (0,3);
  \draw [dotted] (1,0) to (1,1.5) node[morphlabel] {$A$} to (1,3);
\end{tz}
\]
In text form, this is expressed as follows:
\begin{eqnarray*}
  A \otimes B &=& (\id_1 \otimes B) \circ (A \otimes \id_1) \\
  &=& (B \otimes \id_1) \circ (\id_1 \otimes A) \\
  &\xdoubleto{\phi_{B,A}}& (\id_1 \otimes A) \circ (B \otimes \id_1) \\
  &=& B \otimes A
\end{eqnarray*}
The first equation is nudging (Axiom (v)), the second equation is the fact that $1$ is a strict unit (Axiom (iii)), the third is the interchangor, and the fourth is nudging again. It is then a pleasant exercise in the graphical calculus that $\sigma_{A,B}$ is natural and bilinear, so that $(M, \id_1, \odot, \{\sigma_{A,B}\})$ is a symmetric monoidal category. The reverse procedure works in the same way.

\subsection{Equivalence with semistrict monoidal 2-categories}
We can now prove the following.
\begin{proposition} \label{stringentprop} \begin{itemize}
 \item [1.] If $(\bicat{M}, 1, \otimes, \{\Phi_{(f',g'),(f,g)}\})$ is a semistrict monoidal 2-category, then restricting to the underlying interchangor 2-isomorphisms $\phi_{f,g}$ gives a stringent monoidal 2-category. 
 \item [2.] If $(\bicat{M}, 1, \otimes, \{\phi_{f,g}\})$ is a stringent monoidal 2-category, then the interchangor 2-isomorphisms $\phi_{(f,g)}$ can be extended to coherence isomorphisms $\Phi_{(f',g'),(f,g)}$ making $\bicat{M}$ into a semistrict monoidal 2-category.
 \item [3.] The processes in (1) and (2) are inverse to each other, on-the-nose.
\end{itemize}
\end{proposition}
\begin{proof} 1. Axiom (iii) follows since $\otimes$ is a cubical functor, from which it follows that $L_A := A \otimes - $ and $L_B := - \otimes B$ are strict 2-functors.  The equation $L_A L_B = L_{A \otimes B}$ follows from the associativity equation coming from $\bicat{M}$ being a monoid in $\text{2Cat}^{\text{ps}}$. Similarly $R_B R_A = R_{A \otimes B}$, as well as $L_A R_B = R_B L_A$. 

Axiom (v) follows from the cubical identity, as explained above. Axiom (vi) follows from $\bicat{M}$ being a monoid in $\text{2Cat}^\text{ps}$. Axiom (vii) follows from the cubical equation. Axiom (viii) follows from the naturality of $\Phi_{(f',g'), (f,g)}$. Axiom (ix) follows from the coherence equation on $\Phi_{(f',g'), (f,g)}$.

2. We define 
\begin{equation} \label{defn_inter}
\Phi_{(f',g'), (f,g)} := 
\begin{tz}[xscale=2.6, yscale=8] 
\node (1) at (0,0)
{
$\begin{tz}[xscale=0.7, yscale=0.7]
  \draw(0,0) to (0,2.5) node[morphlabel] {$f$} to (0,4) node[morphlabel] {$f'$} to (0,6.5);
  \draw(1,0) to (1,2.5) node[morphlabel] {$g$} to (1,4) node[morphlabel] {$g'$} to (1,6.5);
\end{tz}$
}; 

\node (2) at (1,0)
{
$\begin{tz}[xscale=0.7, yscale=0.7]
  \draw(0,0) to (0,1) node[morphlabel] {$f$} to (0,2.5) node[morphlabel] {$\id$} to (0,4) node[morphlabel] (f') {$f'$} to (0,5.5) node[morphlabel] {$\id$} to (0,6.5);
  \draw(1,0) to (1,1) node[morphlabel] {$\id$} to (1,2.5) node[morphlabel] (g) {$g$} to (1,4) node[morphlabel] {$\id$} to (1,5.5) node[morphlabel] {$g'$} to (1,6.5);
	\draw[green, inner sep=1pt] ([shift={(135:0.1)}] f'.135) rectangle ([shift={(-45:0.1)}] g.-45); 
\end{tz}$
}; 

\node (3) at (2,0)
{
$\begin{tz}[xscale=0.7, yscale=0.7]
  \draw(0,0) to (0,1) node[morphlabel] {$f$} to (0,2.5) node[morphlabel] {$f'$} to (0,4) node[morphlabel] (f') {$\id$} to (0,5.5) node[morphlabel] {$\id$} to (0,6.5);
  \draw(1,0) to (1,1) node[morphlabel] {$\id$} to (1,2.5) node[morphlabel] (g) {$\id$} to (1,4) node[morphlabel] {$g$} to (1,5.5) node[morphlabel] {$g'$} to (1,6.5);
\end{tz}$
}; 

\node (4) at (3,0)
{
$\begin{tz}[yscale=0.7]
  \draw(0,0) to node[morphlabel] {$f' \circ f$} (0,6.5);
  \draw(1,0) to node[morphlabel] {$g' \circ g$} (1,6.5);
 \end{tz}$
}; 

\begin{scope}[double arrow scope]
	\draw (1) -- node[above] {$=$} (2);
	\draw (2) -- node[above] {$\phi_{f', g}$} (3);
	\draw (3) -- node[above] {$=$} (4);
\end{scope}

\end{tz}
\end{equation}

The coherence equation that $\Phi$ must satisfy in order for $\otimes \colon \bicat{M} \times \bicat{M} \rightarrow \bicat{M}$ to be a pseudofunctor looks as follows in wire diagrams:
\begin{equation} \label{coherence}
\begin{tz}[xscale=5, yscale=5]
\node (1) at (0,0)
{
$\begin{tz}[xscale=0.7, yscale=0.7]
  \draw(0,0) to (0,1) node[morphlabel] {$f$} to (0,2.5) node[morphlabel] (f') {$f'$} to (0,4) node[morphlabel] (f'') {$f''$} to (0,5);
  \draw(1,0) to (1,1) node[morphlabel] (g) {$g$} to (1,2.5) node[morphlabel] (g') {$g'$} to (1,4) node[morphlabel] {$g''$} to (1,5);
	\draw[green, inner sep=1pt] ([shift={(135:0.2)}] f''.135) rectangle ([shift={(-45:0.2)}] g'.-45);
	\draw[blue, inner sep=1pt] ([shift={(135:0.2)}] f'.135) rectangle ([shift={(-45:0.2)}] g.-45);
\end{tz}$
}; 

\node (2) at (1,0)
{
$\begin{tz}[yscale=0.7]
  \draw(0,0) to (0,1.5) node[morphlabel] {$f$} to (0,3.5) node[morphlabel] {$f''f$} to (0,5);
  \draw(1,0) to (1,1.5) node[morphlabel] {$g$} to (1,3.5) node[morphlabel] {$g''g$} to (1,5);
\end{tz}$
}; 

\node (3) at (0,-1)
{
$\begin{tz}[xscale=0.7, yscale=0.7]
  \draw(0,0) to (0,1.5) node[morphlabel] {$f'f$} to (0,3.5) node[morphlabel] {$f''$} to (0,5);
  \draw(1,0) to (1,1.5) node[morphlabel] {$g'g$} to (1,3.5) node[morphlabel] {$g''$} to (1,5);
\end{tz}$
}; 

\node (4) at (1,-1)
{
$\begin{tz}[xscale=1.2, yscale=0.7]
  \draw(0,0) to node[morphlabel] {$f'' f' f$} (0,5);
  \draw(1,0) to node[morphlabel] {$g'' g' g$} (1,5);
\end{tz}$
};

\begin{scope}[double arrow scope]
	\draw (1) -- node[above, green] {$\Phi_{(f'',g''),(f',g')}$} (2);
	\draw (1) -- node[left, blue] {$\Phi_{(f',g'),(f,g)}$} (3);
	\draw (2) -- node[right] {$\Phi_{(f'' f', g'' g'), (f, g)}$} (4);
	\draw (3) -- node[below] {$\Phi_{(f'', g''), (f'f, g'g)}$} (4);
\end{scope}
\end{tz}
\end{equation}
It is straightforward to verify graphically that the definition \eqref{defn_inter} of $\Phi$ satisfies the above, using axiom (ix). Also, $\Phi$ is natural because of axiom (viii).  This establishes that $(\otimes, \Phi) \colon \bicat{M} \times \bicat{M} \rightarrow \bicat{M}$ is a pseudofunctor.  It is cubical because of axiom (vii). Associativity follows from axiom (ix). 

3. We need to check that $\Phi$ is uniquely determined by its underlying interchangor 2-isomorphisms $\phi$. This follows from the following commutative diagram:
\[
\begin{tz}[xscale=5, yscale=5.5]

\node (1) at (0,0)
{
$\begin{tz}[xscale=0.7, yscale=0.7]
  \draw(0,0) to (0,1.5) node[morphlabel] {$f$} to (0,3.5) node[morphlabel] (f') {$f'$} to (0,5);
  \draw(1,0) to (1,1.5) node[morphlabel] (g) {$g$} to (1,3.5) node[morphlabel] (g') {$g'$} to (1,5);
\end{tz}$
};

\node (2) at (1,0)
{
$\begin{tz}[xscale=0.7, yscale=0.7]
  \draw(0,0) to (0,1) node[morphlabel] {$f$} to (0,2.5) node[morphlabel] (f') {$f'$} to (0,4) node[morphlabel] (id) {$\id$} to (0,5);
  \draw(1,0) to (1,1) node[morphlabel] (g) {$g$} to (1,2.5) node[morphlabel] (id2) {$\id$} to (1,4) node[morphlabel] {$g'$} to (1,5);
   \draw[green, inner sep=1pt] ([shift={(135:0.2)}] id.135) rectangle ([shift={(-45:0.2)}] id2.-45);
   \draw[blue, inner sep=1pt] ([shift={(135:0.2)}] f'.135) rectangle ([shift={(-45:0.2)}] g.-45);
\end{tz}$
}; 

\node (3) at (2,0)
{
$\begin{tz}[xscale=0.7, yscale=0.7]
  \draw(0,0) to (0,1) node[morphlabel] {$f$} to (0,2.5) node[morphlabel] (id) {$\id$} to (0,4) node[morphlabel] (f') {$f'$} to (0,5.5) node[morphlabel] {$\id$} to (0,6.5);
  \draw(1,0) to (1,1) node[morphlabel] {$\id$} to (1,2.5) node[morphlabel] (g) {$g$} to (1,4) node[morphlabel] {$\id$} to (1,5.5) node[morphlabel] {$g'$} to (1,6.5);
   \draw[green, inner sep=1pt] ([shift={(135:0.2)}] f'.135) rectangle ([shift={(-45:0.2)}] g.-45);
\end{tz}$
}; 

\node (4) at (0,-1)
{
$\begin{tz}[xscale=0.7, yscale=0.7]
  \draw(0,0) to node[morphlabel] {$f'f$} (0,5);
  \draw(1,0) to node[morphlabel] {$g'g$} (1,5);
\end{tz}$
}; 

\node (5) at (1,-1)
{
$\begin{tz}[xscale=0.7, yscale=0.7]
  \draw(0,0) to (0,1.5) node[morphlabel] {$f'f$} to (0,3.5) node[morphlabel] (f') {$\id$} to (0,5);
  \draw(1,0) to (1,1.5) node[morphlabel] {$g$} to (1,3.5) node[morphlabel] (g') {$g'$} to (1,5);
\end{tz}$
}; 

\node (6) at (2,-1)
{
$\begin{tz}[xscale=0.7, yscale=0.7]
  \draw(0,0) to (0,1) node[morphlabel] {$f$} to (0,2.5) node[morphlabel] (f') {$f'$} to (0,4) node[morphlabel] (id) {$\id$} to (0,5);
  \draw(1,0) to (1,1) node[morphlabel] (id2) {$\id$} to (1,2.5) node[morphlabel] (g) {$g$} to (1,4) node[morphlabel] {$g'$} to (1,5);
   \draw[green, inner sep=1pt] ([shift={(135:0.2)}] f'.135) rectangle ([shift={(-45:0.2)}] id2.-45);
\end{tz}$
}; 

\draw[double] (2) -- (3);

\begin{scope}[double arrow scope]
	\draw (2) -- node[above, green] {$\Phi_{(\id,g'),(f',\id)}$} (1);
	\draw (1) -- node[right] {$\Phi_{(f',g'),(f,g)}$} (4);
	\draw (3) -- node[left] {$\Phi_{(f',\id),(\id,g)}$} (6);
	\draw (6) -- node[below] {$\Phi_{(f',g),(f, \id)}$} (5);
	\draw (5) -- node[below] {$\Phi_{(\id,g'),(f'f, g)}$} (4);
	\draw (2) -- node[right, blue] {$\Phi_{(f', \id),(f, g)}$} (5);
	
\end{scope}
\end{tz}
\]
Each square is an instance of \eqref{coherence} and hence commutes. Due to the cubical condition, all the $\Phi$ terms are the identity except for the one on the far right, hence we have $\Phi_{(f',g'),(f,g)} = \phi_{f',g}$. In other words, the formula \eqref{defn_inter} holds in every semistrict monoidal 2-category. 
\end{proof}

\section{Stringent symmetric monoidal 2-categories} \label{stringent_sym_sec}
In this section we define {\em stringent symmetric monoidal 2-categories}, and extend the wire diagram calculus to them. 

\subsection{The definition}

\begin{defn} A {\em stringent symmetric monoidal 2-category} $\left(\bicat{M}, \, 1, \, \otimes, \, \{\phi_{f,g} \}, \, \{\beta_{A,B} \} \right)$ is a stringent monoidal 2-category $\left(\bicat{M}, 1, \otimes \{\phi_{f,g}\}\right)$ equipped with, for every pair of objects $A, B \in \mathcal{M}$, a 1-morphism
\[
\beta_{A,B} \colon A \otimes B \rightarrow B \otimes A, \quad \text{drawn as } \,\, \begin{tz}[xscale=0.7]  \draw(0,0) node[below] {$A$} to[out=up, in=down] (1,1) node[above]{B};  \draw (1,0) node[below] {$B$} to[out=up, in=down] (0,1) node[above] {$A$}; \end{tz}
\]
satisfying the following equations between 1-morphisms {\em on-the-nose}:
\begin{itemize}
\item[(i)] $\beta_{A,B} \circ \beta_{B,A} = \id_{A \otimes B}$. In wire diagrams:
\[
\begin{tz}[xscale = 0.7]
	\draw (0,0) node[below] {$A$} to[out=up, in=down] (1,1) to[out=up, in=down] (0,2) node[above] {$A$};
	\draw (1,0) node[below] {$B$} to[out=up, in=down] (0,1) to[out=up, in=down] (1,2) node[above] {$B$};
\end{tz}
\,\,=\,\,
\begin{tz}[xscale = 0.7]
	\draw (0,0) node[below] {$A$} to (0,2) node[above] {$A$};
	\draw (1,0) node[below] {$B$} to (1,2) node[above] {$B$};
\end{tz}
\]
\item[(ii)]  $\beta_{A \otimes B, C} = (\beta_{(A,C)} \otimes \id) \circ (\id \otimes \beta_{B, C})$. In wire diagrams:
\[
\begin{tz}[xscale=0.7]
	\draw (2,0) node[below] {$C$} to[out=up, in=down] (0,2) node[above] {$C$};
	\draw[fill=white, draw=white] (0.5, 1.2) rectangle (1.5, 0.8);
	\draw (0,0) node[below] {$A$} to[out=up, in=down] (1,2) node[above] {$A$};
	\draw (1,0) node[below] {$B$} to[out=up, in=down] (2,2) node[above] {$B$};
\end{tz}
\,\, = \,\,
\begin{tz}[xscale=0.7]
	\draw (2,0) node[below] {$C$} to[out=up, in=down] (0,2) node[above] {$C$};
	\draw (0,0) node[below] {$A$} to[out=up, in=down] (1,2) node[above] {$A$};
	\draw (1,0) node[below] {$B$} to[out=up, in=down] (2,2) node[above] {$B$};
\end{tz}
\]
There is a similar equation for $\beta_{A, B \otimes C}$.
\item[(iii)] If $f \colon A \rightarrow A'$ is a 1-morphism, then $\beta_{A', B} \circ (f \otimes \id_B) = (\id_{A} \otimes f) \circ \beta_{A,B}$. In wire diagrams:
\[
\begin{tz}[xscale=0.7]
	\draw (0,0) node[below] {$A$} to node[morphlabel] {$f$} (0,1) to[out=up, in=down] (1,2) to (1,3) node[above] {$A'$};
	\draw (1,0) node[below] {$B$} to (1,1) to[out=up, in=down] (0,2) to (0,3) node[above] {$B$};
\end{tz} \,\, = \,\,
\begin{tz}[xscale=0.7]
	\draw (0,0) node[below] {$A$} to (0,1) to[out=up, in=down] (1,2) to node[morphlabel] {$f$}  (1,3) node[above] {$A'$};
	\draw (1,0) node[below] {$B$} to (1,1) to[out=up, in=down] (0,2) to (0,3) node[above] {$B$};
\end{tz}
\]
There is a similar equation for $g \colon B \rightarrow B'$.
\end{itemize}
Moreover, we require the following equation between 2-morphisms:
\begin{itemize}
 \item[(iv)] For every 1-morphism $f \colon A \rightarrow A'$ and every pair of objects $B,C$, $\phi_{f, \beta_{B,C}} = \id$. In wire diagrams:
 \[
\begin{tz}[xscale=3] 
\node(1) at (0,0)
{
$\begin{tz}[xscale=0.7]
	\draw (0,0) node[below] {$A$} to (0, 1) to node[morphlabel] {$f$} (0,2) node[above] {$A'$};
	\draw (1,0) node[below] {$B$} to[out=up, in=down] (2, 1) to (2,2) node[above] {$B$};
	\draw (2,0) node[below] {$C$} to[out=up, in=down] (1,1) to (1,2) node[above] {$C$};
\end{tz}$
};

\node(2) at (1,0)
{
$\begin{tz}[xscale=0.7]
	\draw (0,0) node[below] {$A$} to node[morphlabel] {$f$} (0,1) to (0,2) node[above] {$A'$};
	\draw (1,0) node[below] {$B$} to (1,1) to[out=up, in=down] (2, 2) node[above] {$B$};
	\draw (2,0) node[below] {$C$} to (2,1) to[out=up, in=down] (1,2) node[above] {$C$};
\end{tz}$
};

\begin{scope}[double arrow scope]
	\draw (1) -- node[above] {$\phi_{f, \beta_{B,C}}$} (2);
\end{scope}

\end{tz}
\,\, = \,\,
\begin{tz}[xscale=3] 
\node(1) at (0,0)
{
$\begin{tz}[xscale=0.7]
	\draw (0,0) node[below] {$A$} to (0, 1) to node[morphlabel] {$f$} (0,2) node[above] {$A'$};
	\draw (1,0) node[below] {$B$} to[out=up, in=down] (2, 1) to (2,2) node[above] {$B$};
	\draw (2,0) node[below] {$C$} to[out=up, in=down] (1,1) to (1,2) node[above] {$C$};
\end{tz}$
};

\node(2) at (1,0)
{
$\begin{tz}[xscale=0.7]
	\draw (0,0) node[below] {$A$} to node[morphlabel] {$f$} (0,1) to (0,2) node[above] {$A'$};
	\draw (1,0) node[below] {$B$} to (1,1) to[out=up, in=down] (2, 2) node[above] {$B$};
	\draw (2,0) node[below] {$C$} to (2,1) to[out=up, in=down] (1,2) node[above] {$C$};
\end{tz}$
};

\begin{scope}[double arrow scope]
	\draw (1) -- node[above] {$\id$} (2);
\end{scope}

\end{tz}
\] 
Similarly, $\phi_{\beta_{A,B}, g} = \id$ for every $g \colon C \rightarrow C'$.
\end{itemize}
\end{defn}
Having given the definition, note that the naturality condition (iii) on $\beta_{A,B}$ does not hold for tensor products! In general, we need to insert the interchangor $\phi$ in order to commute $f\otimes g$ past the braiding:
\begin{equation} \label{multi-braid}
\begin{tz}[xscale=0.7]
	\draw (0,0) node[below] {$A$} to node[morphlabel] {$f$} (0,1) to[out=up, in=down] (1,2) to (1,3) node[above] {$A'$};
	\draw (1,0) node[below] {$B$} to node[morphlabel] {$g$} (1,1) to[out=up, in=down] (0,2) to (0,3) node[above] {$B'$};
\end{tz}
\,\, = \,\,
\begin{tz}[xscale=0.7]
	\draw (0,0) node[below] {$A$} to (0, 0.3) node[morphlabel] {$f$} to (0,1) to[out=up, in=down] (1,2) to (1,3) node[above] {$A'$};
	\draw (1,0) node[below] {$B$} to (1, 0.8) node[morphlabel] {$g$} to (1,1) to[out=up, in=down] (0,2) to (0,3) node[above] {$B'$};
\end{tz}
\,\, = \,\,
\begin{tz}[xscale=0.7]
	\draw (0,0) node[below] {$A$} to (0, 0.3) node[morphlabel] {$f$} to (0,1) to[out=up, in=down] (1,2) to (1,3) node[above] {$A'$};
	\draw (1,0) node[below] {$B$} to (1, 0.7) to (1,1) to[out=up, in=down] (0,2) to node[morphlabel] {$g$} (0,3) node[above] {$B'$};
\end{tz}
\,\, = \,\,
\begin{tz}[xscale=0.7]
	\draw (0,0) node[below] {$A$} to (0,1) to[out=up, in=down] (1,2) to (1,2.2) node[morphlabel] {$f$} to (1,3) node[above] {$A'$};
	\draw (1,0) node[below] {$B$} to (1,1) to[out=up, in=down] (0,2) to (0,2.7) node[morphlabel] {$g$} to (0,3) node[above] {$B'$};
\end{tz}
\,\, \stackrel{\raisebox{0.4em}{$\phi_{g,f}$}}{\Longrightarrow} \,\,
\begin{tz}[xscale=0.7]
	\draw (0,0) node[below] {$A$} to (0,1) to[out=up, in=down] (1,2) to (1,2.5) node[morphlabel] {$f$} to (1,3) node[above] {$A'$};
	\draw (1,0) node[below] {$B$} to (1,1) to[out=up, in=down] (0,2) to (0,2.5) node[morphlabel] {$g$} to (0,3) node[above] {$B'$};
\end{tz}
\end{equation}

\subsection{Example} This example is adapted from~\cite[Example 2.30]{CSPthesisLatest}; see also~~\cite{gk14-sep,hs05-qf, s12-hthc, jo12-mso}. Let $\mathbb{S}$ be the {\em sphere spectrum}, so that $\pi_i(\mathbb{S})$ are the stable homotopy groups of spheres:
\[
 \pi_0^\text{st} = \mathbb{Z}, \quad \pi_1^\text{st} = \mathbb{Z}/2, \quad \pi_2^\text{st} = \mathbb{Z}/2, \quad \pi_3^\text{st}=\mathbb{Z}/24, \quad \ldots
\]
We can conceive of the truncation $\mathbb{S}_{[0,2]}$ for $0 \leq i \leq 2$ as a quasistrict symmetric monoidal 2-category \bicat{Q}, as follows. The objects of \bicat{Q} are the integers $\mathbb{Z}$. The hom-categories are given by
\[
 \Hom_\bicat{Q} (m, n) = \begin{cases} P & \text{ if $m$=$n$ } \\ \text{empty} & \text{ otherwise.} \end{cases}
\]
Here, $P$ is a skeletal version of the Picard category $\text{Pic}^{\mathbb{Z}/2} (\mathbb{Z})$ whose objects are $\mathbb{Z}/2$-graded  free abelian groups of total rank 1 and whose morphisms are invertible graded homomorphisms, with the usual $\mathbb{Z}/2$-graded tensor product and the Koszul rule for the symmetry~\cite{gk14-sep}. So, $P$ has two objects $0$ and $1$, and each object has two automorphisms, $I$ and $-I$. The tensor product in $P$ is given on objects by addition mod 2, and on morphisms by multiplication. The braiding $b$ on the symmetric monoidal category $P$ is given by the Koszul rule, with the only nonidentity braiding given by $b_{1, 1} = -I$.

We reindex $P$ so as to form part of the 2-category $\bicat{Q}$. So, composition of 1-morphisms in $\mathbb{Q}$ corresponds to tensor product inside $P$.

The tensor product on $\bicat{Q}$ is given on objects by addition in $\mathbb{Z}$, and on 1- and 2-morphisms by tensor product in $P$. Let $(m,i)$ be $(n,j)$ with $m,n \in \mathbb{Z}$ and $i,j \in \{0,1\}$ be automorphisms of $m$ and $n$ in $\bicat{B}$ respectively. The interchangor 
\[
 \phi_{(m, i), (n, j)} \colon i + j  \rightarrow i + j
\]   
is defined to be the braiding $b_{i,j}$ inside the symmetric monoidal category $P$.  These constructions equip $\bicat{Q}$ as a stringent monoidal 2-category.

The braiding 1-morphisms 
\[
  \beta_{m,n} \colon m+n \rightarrow m+n
\]
in $\bicat{Q}$ are defined as $\beta_{m,n} = m + n$ (mod $2$). This completes the description of $\bicat{Q}$ as a stringent symmetric monoidal 2-category.

\subsection{Equivalence with semistrict symmetric monoidal 2-categories}
We now recall the definition of a quasistrict symmetric monoidal 2-category, and prove that they are equivalent to stringent ones. The following definition is taken from~\cite{CSPthesisLatest, cr98-gcb} and builds on the definition of a symmetric monoidal bicategory from~\cite{CSPthesisLatest, st13-ccb}. The reader is referred to these references for the the definitions of the braiding `bilinearators' $R_{A,B|C}$ and $S_{A | B,C}$ etc. 

\begin{defn}[~\cite{CSPthesisLatest, cr98-gcb}] A {\em Crans semistrict symmetric monoidal 2-category} is a symmetric monoidal bicategory $\bicat{M}$ such that:
\begin{itemize}
 \item[(CSS.1)] The underlying monoidal bicategory is a Gray monoid; 
 \item[(CSS.2)] The following additional normalization conditions apply:
  \begin{itemize}
   \item [(a)] The 1-morphisms $\beta_{1, x}$ and $\beta_{x,1}$ are identity morphisms on $A$, for each object $A \in \bicat{M}$.
   \item [(b)] The isomorphisms $R_{1,A|B}$, $R_{A,1|B}$, $S_{A|1,B}$, and $S_{A|B,1}$ are the identity 2-isomorphism of $\beta_{A,B}$.
 \item[(c)] The isomorphisms $R_{A,B|1}$ and $S_{1|A,B}$ are the identity 2-isomorphisms of $I_{A \otimes B}$.  
  \end{itemize}
\end{itemize}
\end{defn}
A Crans semistrict symmetric monoidal 2-category comes equipped with 2-isomorphisms
\begin{equation} \label{symettrators}
 \sigma_{A,B} \colon \id_{A \otimes B} \Rightarrow \beta_{B,A} \circ \beta_{A,B}
\end{equation}
for each pair of objects $A,B \in \bicat{M}$, which witness the fact that the braiding is symmetric. 
\begin{defn}[~\cite{CSPthesisLatest}] A {\em quasistrict symmetric monoidal 2-category} is a Crans semistrict symmetric monoidal 2-category $\bicat{M}$ such that:
\begin{itemize}
 \item [(QS.1)] The modifications $R$, $S$, and $\sigma$ are identities. 
 \item [(QS.2)] For the transformation $\beta = (\beta_{A,B}, \, \beta_{f,g})$, the component $\beta_{f,g}$ is the identity if either $f$ or $g$ is an identity morphism.
\item[(QS.3)] The 2-morphism witnessing naturality:
 \[
  \Phi_{(f',g'), (f,g)} : (f' \otimes g') \circ (f \otimes g) \Rightarrow (f' \circ f) \otimes (g' \circ g)
 \]
is an identity if either $f'$ or $g$ is a component of $\beta$, i.e. if $f' =\beta_{A,B}$ or $g = \beta_{A,B}$, for some pair of objects $A,B \in \bicat{M}$.
\end{itemize}
\end{defn}
\noindent Recall Theorem \ref{csp_qs_theorem} of Schommer-Pries from the Introduction:
\theoremstyle{plain}
\newtheorem*{thm:csp-theorem}{Theorem \ref{csp_qs_theorem}}
\begin{thm:csp-theorem}[~\cite{CSPthesisLatest}] Every symmetric monoidal bicategory is equivalent to a quasistrict symmetric monoidal 2-category. 
\end{thm:csp-theorem}
We now show that quasistrict symmetric monoidal 2-categories are equivalent to stringent symmetric monoidal 2-categories. 
\begin{proposition} \label{stringentsymmprop} \begin{itemize}
 \item [1.] If $\left(\bicat{M}, \, 1, \, \otimes, \, \{\Phi_{(f',g'),(f,g)}\}, \, \{\beta_{A,B}\}, \{\beta_{f,g}\}\right)$ is a quasistrict monoidal 2-category, then restricting to the underlying interchangor 2-isomorphisms $\phi_{f,g}$ gives a stringent symmetric monoidal 2-category $\left(\bicat{M}, \, 1, \, \otimes, \, \{\phi_{f,g} \}, \, \{\beta_{A,B} \} \right)$. 
 \item [2.] If $\left(\bicat{M}, \, 1, \, \otimes, \, \{\phi_{f,g} \}, \, \{\beta_{A,B} \} \right)$ is a stringent symmetric monoidal 2-category, then coherence isomorphisms $\beta_{f,g}$ can be introduced, and the interchangor 2-isomorphisms $\phi_{f,g}$ can be extended to coherence isomorphisms $\Phi_{(f',g'),(f,g)}$, so as to make $\bicat{M}$ into a quasistrict symmetric monoidal 2-category.
 \item [3.] The processes in (1) and (2) are inverse to each other, on-the-nose.
\end{itemize}
\end{proposition}
\begin{proof}
1. The assertion that $\sigma$ is the identity gives Axiom (i) of a stringent symmetric monoidal 2-category.  Similarly the assertion that $R$ and $S$ are identities gives Axiom (ii). Axiom (iii) follows from (QS.2), and Axiom (iv) follows from (QS.3). 
\vskip 0.2cm
2. We have already defined how to extend $\phi_{f,g}$ to $\Phi_{(f',g'),(f,g)}$ in \eqref{defn_inter}. We define $\beta_{f,g}$ by running \eqref{multi-braid} in reverse. That is, we define
\[
\begin{tz}[xscale=3]

\node (1) at (0,0)
{
$\begin{tz}[xscale=0.7]
	\draw (0,0) node[below] {$A$} to (0,1) to[out=up, in=down] (1,2) to (1,2.5) node[morphlabel] {$f$} to (1,3) node[above] {$A'$};
	\draw (1,0) node[below] {$B$} to (1,1) to[out=up, in=down] (0,2) to (0,2.5) node[morphlabel] {$g$} to (0,3) node[above] {$B'$};
\end{tz}$
}; 

\node (2) at (1,0)
{
$\begin{tz}[xscale=0.7]
	\draw (0,0) node[below] {$A$} to node[morphlabel] {$g$} (0,1) to[out=up, in=down] (1,2) to (1,3) node[above] {$A'$};
	\draw (1,0) node[below] {$B$} to node[morphlabel] {$f$} (1,1) to[out=up, in=down] (0,2) to (0,3) node[above] {$B'$};
\end{tz}$
}; 

\begin{scope}[double arrow scope]
	\draw (1) -- node[above] {$\beta_{f,g}$} (2);
\end{scope}
\end{tz}
\]  
as the following composite:
\begin{equation} \label{multi-braid2}
\beta_{f,g} := \,
\begin{tz}[xscale=0.7]
	\draw (0,0) node[below] {$A$} to (0,1) to[out=up, in=down] (1,2) to (1,2.5) node[morphlabel] {$g$} to (1,3) node[above] {$A'$};
	\draw (1,0) node[below] {$B$} to (1,1) to[out=up, in=down] (0,2) to (0,2.5) node[morphlabel] {$f$} to (0,3) node[above] {$B'$};
\end{tz}
\,\, \stackrel{\raisebox{0.4em}{$\phi_{f,g}^{-1}$}}{\Longrightarrow} \,\,
\begin{tz}[xscale=0.7]
	\draw (0,0) node[below] {$A$} to (0,1) to[out=up, in=down] (1,2) to (1,2.2) node[morphlabel] {$g$} to (1,3) node[above] {$A'$};
	\draw (1,0) node[below] {$B$} to (1,1) to[out=up, in=down] (0,2) to (0,2.7) node[morphlabel] {$f$} to (0,3) node[above] {$B'$};
\end{tz}
\,\, = \,\,
\begin{tz}[xscale=0.7]
	\draw (0,0) node[below] {$A$} to (0, 0.3) node[morphlabel] {$g$} to (0,1) to[out=up, in=down] (1,2) to (1,3) node[above] {$A'$};
	\draw (1,0) node[below] {$B$} to (1, 0.7) to (1,1) to[out=up, in=down] (0,2) to node[morphlabel] {$f$} (0,3) node[above] {$B'$};
\end{tz}
\,\, = \,\,
\begin{tz}[xscale=0.7]
	\draw (0,0) node[below] {$A$} to (0, 0.3) node[morphlabel] {$g$} to (0,1) to[out=up, in=down] (1,2) to (1,3) node[above] {$A'$};
	\draw (1,0) node[below] {$B$} to (1, 0.8) node[morphlabel] {$f$} to (1,1) to[out=up, in=down] (0,2) to (0,3) node[above] {$B'$};
\end{tz}
\,\, = \,\,
\begin{tz}[xscale=0.7]
	\draw (0,0) node[below] {$A$} to node[morphlabel] {$g$} (0,1) to[out=up, in=down] (1,2) to (1,3) node[above] {$A'$};
	\draw (1,0) node[below] {$B$} to node[morphlabel] {$f$} (1,1) to[out=up, in=down] (0,2) to (0,3) node[above] {$B'$};
	\end{tz}
\end{equation}
It is now routine to show that $\beta_{f,g}$ satisfies all the coherence equations listed in~\cite{CSPthesisLatest} for a quasistrict symmetric monoidal 2-category. Indeed, these equations can be translated into wire diagrams and the proof is entirely graphical. In particular, (QS.1) implies Axioms (i) and (ii), (QS.2) implies Axiom (iii), and (QS.3) implies Axiom (iv). 
\vskip 0.2cm

3. We need to show that $\beta_{f,g}$ is uniquely determined as the composite \eqref{multi-braid2}. Now, $\beta_{f,g}$ are the coherence 2-isomorphisms coming from the fact that $\beta$ is a transformation $\beta \colon \otimes \Rightarrow \otimes \circ \text{swap}$. Hence they satisfy the following coherence equation:
\begin{equation} \label{trans_cohere}
\begin{tz}[xscale=4.5, yscale=6]

\node (1) at (0,0)
{
$\begin{tz}[xscale=1.2]
	\draw (0,0) node[below] {$A$} node (bl) {} to[out=up, in=down] (1,1) -- node[morphlabel] (g) {$g$} (1,2) -- node[morphlabel] (g') {$g'$} (1,3) node[above] {$A''$};
	\draw (1,0) node[below] {$B$} node (br) {} to[out=up, in=down] (0,1) -- node[morphlabel] (f) {$f$} (0,2) -- node[morphlabel] (f') {$f'$} (0,3) node[above] {$B''$};
 	\draw[red, inner sep=1pt] ([shift={(135:0.2)}] f'.135) rectangle ([shift={(-45:0.2)}] g.-45);
 	\draw[green, inner sep=1pt] ([shift={(135:0.3)}] f.135) rectangle ([shift={(-45:0.3)}] 1.2,0.1);
\end{tz}$
};  

\node (2) at (1,0)
{
$\begin{tz}[xscale=1.2]
	\draw (0,0) node[below] {$A$} to[out=up, in=down] (1,1) -- node[morphlabel] (B){$g' \circ g$} (1,3) node[above] {$A''$};
	\draw (1,0) node[below] {$B$} to[out=up, in=down] (0,1) -- node[morphlabel] (A) {$f' \circ f$} (0,3) node[above] {$B''$};
\end{tz}$
}; 

\node (3) at (2,0)
{
$\begin{tz}[xscale=1.2]
	\draw (0,0) node[below] {$A$} -- node[morphlabel] {$g' \circ g$} (0,2) to[out=up, in=down] (1,3) node[above] {$A''$};
	\draw (1,0) node[below] {$B$} -- node[morphlabel] {$f' \circ f$} (1,2) to[out=up, in=down] (0, 3) node[above] {$B''$};
\end{tz}$
}; 

\node (4) at (0.5,-1)
{
$\begin{tz}[xscale=1.2]
	\draw (0,0) node[below] {$A$} -- node[morphlabel] {$g$} (0,1) to[out=up, in=down] (1,2) -- node[morphlabel] (g') {$g'$} (1,3) node[above] {$A''$};
	\draw (1,0) node[below] {$B$} -- node[morphlabel] {$f$} (1,1) to[out=up, in=down] (0,2) -- node[morphlabel] (f') {$f'$} (0,3) node[above] {$B''$};
 	\draw[green, inner sep=1pt] ([shift={(135:0.2)}] f'.north west) rectangle ([shift={(0:0.2)}] g'.north east |- 1,1);
\end{tz}$
}; 

\node (5) at (1.5,-1)
{
$\begin{tz}[xscale=1.2]
	\draw (0,0) node[below] {$A$} -- node[morphlabel] {$g$} (0,1) -- node[morphlabel] (g') {$g'$} (0,2) to[out=up, in=down] (1,3) node[above] {$A''$};
	\draw (1,0) node[below] {$B$} -- node[morphlabel] (f) {$f$} (1,1) -- node[morphlabel] {$f'$} (1,2) to[out=up, in=down] (0,3) node[above] {$B''$};
 	\draw[green, inner sep=1pt] ([shift={(135:0.2)}] g'.north west) rectangle ([shift={(-45:0.2)}] f.south east);
\end{tz}$
};

\begin{scope}[double arrow scope]
	\draw (1) -- node[above, red] {$\Phi_{(f',g'),(f,g)}$} (2);
	\draw (2) -- node[above] {$\beta_{f'f, g'g}$} (3);
	\draw (1) -- node[below left, green] {$\beta_{f,g}$} (4);
	\draw (4) -- node[below] {$\beta_{f',g'}$} (5);
	\draw (5) -- node[below right] {$\Phi_{(g', f'), (g,f)}$} (3);
\end{scope}

\end{tz}
\end{equation}
In \eqref{trans_cohere}, set $g' = \id$ and $f=\id$. Then, using $\beta_{id, g} = \id$ and $\beta_{f', \id} = \id$, we obtain precisely the formula \eqref{multi-braid2} for $\beta_{f,g}$.
\end{proof}

\bibliography{references}

\begin{thebibliography}{10}

\bibitem{bl98-2t}
John Baez and Laurel Langford.
\newblock 2-{T}angles.
\newblock {\em Letters in Mathematical Physics}, 43:187--197, 1998.

\bibitem{BN96}
John~C. Baez and Marin Neuchl.
\newblock Higher-dimensional algebra {I}. {B}raided monoidal 2-categories.
\newblock {\em Advances in Mathematics}, 121(2):196--244, 1996.

\bibitem{PaperIII}
Bruce Bartlett, Christopher Douglas, Christoper Schommer-Pries, and Jamie
  Vicary.
\newblock A classification of 3-dimensional topological quantum field theories
  extended to 1-manifolds.
\newblock In preparation.

\bibitem{PaperII}
Bruce Bartlett, Christopher Douglas, Christoper Schommer-Pries, and Jamie
  Vicary.
\newblock Extended 3-dimensional bordisms as the theory of modular objects.
\newblock In preparation.

\bibitem{PaperI}
Bruce Bartlett, Christopher Douglas, Christoper Schommer-Pries, and Jamie
  Vicary.
\newblock A finite presentation of the 3-dimensional bordism bicategory.
\newblock In preparation.

\bibitem{cr98-gcb}
Sjoerd~E. Crans.
\newblock Generalized centers of braided and sylleptic monoidal 2-categories.
\newblock {\em Advances in Mathematics}, 136(2):183--223, 1998.

\bibitem{ds97-mbh}
Brian Day and Ross Street.
\newblock Monoidal bicategories and hopf algebroids.
\newblock {\em Advances in Mathematics}, 129(1):99--157, 1997.

\bibitem{gk14-sep}
Nora Ganter and Mikhail Kapranov.
\newblock Symmetric and exterior powers of categories.
\newblock {\em Transformation Groups}, 19(1):57--103, 2014.
\newblock Also available as
  \href{http://arxiv.org/abs/1110.4753}{arXiv:1110.4753}.

\bibitem{Gr74}
John~W. Gray.
\newblock {\em Formal Category Theory: Adjointness for 2-Categories}.
\newblock Number 391 in Lecture Notes in Mathematics. Springer, 1974.

\bibitem{Gr76}
John~W. Gray.
\newblock Coherence for the tensor product of 2-categories, and braid groups.
\newblock In S.~Eilenberg A.~Heller, M.~Tierney, editor, {\em Algebra, Topology
  and Category Theory: a collection in honour of Samuel Eilenberg}, pages
  63--76. Academic Press, 1976.

\bibitem{gurskithesis}
Nick Gurski.
\newblock {\em An algebraic theory of tricategories}.
\newblock PhD thesis, University of Chicago, 2006.
\newblock Available online at
  \href{http://www.math.yale.edu/?mg622/tricats.pdf}{http://www.math.yale.edu/?mg622/tricats.pdf}.

\bibitem{gur11-lsc}
Nick Gurski.
\newblock Loop spaces, and coherence for monoidal and braided monoidal
  bicategories.
\newblock {\em Advances in Mathematics}, 5:4225--4265, 2011.

\bibitem{g13-ctd}
Nick Gurski.
\newblock {\em Coherence in three-dimensional category theory}, volume 201 of
  {\em Cambridge tracts in Mathematics}.
\newblock Cambridge University Press, Cambridge, 2013.

\bibitem{go13-ils}
Nick Gurski and Ang\'{e}lica~M. Osorno.
\newblock Infinite loop spaces, and coherence for symmetric monoidal
  bicategories.
\newblock {\em Advances in Mathematics}, 246:1--32, 2013.

\bibitem{hs05-qf}
Michael~J. Hopkins and Isadore~M. Singer.
\newblock Quadratic functions in geometry, topology, and {M}-theory.
\newblock {\em Journal of Differential Geometry}, 70(3):329--452, 2005.

\bibitem{bms13-gdd}
Catherine~Meusberger John~Barrett and Gregor Schaumann.
\newblock Gray categories with duals and their diagrams.
\newblock 2006.
\newblock \href{http://arxiv.org/abs/1211.0529}{arXiv:1211.0529}.

\bibitem{jo12-mso}
Niles Johnson and Ang\'{e}lica Osorno.
\newblock Modeling stable one-types.
\newblock {\em Theory and Applications of Categories}, 26(20):520--537, 2012.

\bibitem{kv94-bm2}
M.~Kapranov and V.~Voevodsky.
\newblock Braided monoidal 2-categories and {M}anin-{S}chechtman higher braid
  groups.
\newblock {\em Journal of Pure and Applied Algebra}, 92(3):241--267, 1994.

\bibitem{kv94-2categories}
Mikhail Kapranov and Vladimir Voevodsky.
\newblock 2-categories and {Z}amolodchikov tetrahedra equations.
\newblock {\em Proc. Sympos. Pure Math}, 56(2):177--259, 1994.

\bibitem{L98}
Tom Leinster.
\newblock Basic bicategories.
\newblock 1998.
\newblock \href{http://arxiv.org/abs/math.CT/9810017}{arXiv:math/9810017}.

\bibitem{m00-bc}
Paddy McCrudden.
\newblock Balanced coalgebroids.
\newblock {\em Theory and Applications of Categories}, (6):71--147, 2000.

\bibitem{GPS95}
A.~J.~Power R.~Gordon and R.~Street.
\newblock Coherence for tricategories.
\newblock {\em Memoirs the American Mathematical Society}, 117, 1995.

\bibitem{CSPthesisLatest}
Christopher Schommer-Pries.
\newblock {\em The Classification of Two-Dimensional Extended Topological Field
  Theories}.
\newblock PhD thesis, Department of Mathematics, University of California,
  Berkeley, 2009.
\newblock Revised version from 2014. Available at
  \href{http://arxiv.org/abs/1112.1000v2}{arXiv:1112.1000v2}.

\bibitem{s12-hthc}
Carlos Simpson.
\newblock {\em Homotopy theory of higher categories}, volume~19 of {\em New
  Mathematical Monographs}.
\newblock Cambridge University Press, 2012.

\bibitem{st13-ccb}
Michael Stay.
\newblock Compact closed bicategories, 2013.
\newblock Available at \href{http://arxiv.org/abs/1301.1053}{arXiv:1301.1053}.

\end{thebibliography}
\bibliographystyle{plain}

\end{document}